\documentclass[a4paper,11pt]{article}
\usepackage{amsmath}
\usepackage{amsthm}
\usepackage{amsfonts}
\usepackage{amssymb}
\usepackage{indentfirst}
\usepackage[left=2.5cm,top=2.5cm,right=2.5cm,bottom=2.5cm]{geometry}

\newtheorem{lem}{Lemma}[section]
\newtheorem{thm}{Theorem}[section]
\newtheorem{co}[thm]{Corollary}
\newtheorem{pr}[thm]{Proposition}

\newtheorem{hyp}[thm]{Assumption}

\numberwithin{equation}{section}

\newcommand{\R}{\mathbb{R}}

\newcommand{\T}{\mathbb{T}}
\newcommand{\Q}{\mathbb{Q}}
\newcommand{\p}{\mathbb{P}}
\begin{document}
\title{Weak law of large numbers for some Markov chains along non homogeneous genealogies}
\author{Vincent Bansaye \footnote{CMAP, Ecole Polytechnique, Palaiseau, France} 
, Chunmao Huang  \footnote{ CMAP, Ecole Polytechnique, Palaiseau, France. Corresponding author at sasamao02@gmail.com.} 
 }
\maketitle
\vspace{1.5cm}
\begin{abstract}
We consider a population with non-overlapping generations, whose size goes to infinity. It is described by a discrete genealogy which may be time non-homogeneous and we pay special attention 
to  branching trees in varying environments. A Markov chain  models the dynamic of the trait of each individual  along this genealogy and may also be  time non-homogeneous. 
Such models are motivated by  transmission processes in the  cell division, reproduction-dispersion dynamics %(Branching Markov Chains) 
or sampling problems in evolution.
We want to determine the evolution of the distribution of the traits among the population, namely the asymptotic behavior of the proportion of individuals with a given trait. We prove some quenched laws of large numbers which rely on the ergodicity of an auxiliary process, in the same vein as \cite{guy,delmar}. Applications to time inhomogeneous Markov chains
 lead us to derive a backward (with respect to the environment) law of large numbers and  a law of large numbers on the whole population until generation $n$. A central limit is also established in the transient case. 
\end{abstract}


\textit{MSC classes.} 60J05, 60J20,	60J80, 60J85, 60F05.

 \bigskip

\section{Introduction}\label{LLS1}

In this paper, we consider Markov chains which are indexed by  discrete trees. This approach is motivated in particular  by the study of structured populations. The tree is thus describing the genealogy of the population in discrete time, with non overlapping generations and the nodes of the tree are the individuals.  We consider  a trait in the population, which could be the location of the individual, its phenotype, its genotype or any biological characteristic. Letting this trait evolve as a Markov chain and be transmitted to the offspring with a random transition leads us to consider a Markov chain indexed by the genealogical tree.
Such a process can also be regarded as a branching particle system where the  offspring of each particle is given by the genealogy and the associated traits by the Markov chain. \\

%\emph{Don't we need anything more on this space $\mathcal X$ ? Guyon requires it is a metric space. At least we will have to precise the arrival state space of the function $f$ we use. Are we reduced to $\R$, or can we take $\R^d$ or something more adapated ? I hav changed that hereafter}

%\textbf{R: Maybe the measurable space is enough, if all the definitions and calculations on the space are well-defined. The arrival state of function is restrict to $\R$.}

Let ($\mathcal X, B_{\mathcal X}$) be a measurable space. 
 The process starts with an initial single individual $\emptyset$ with trait $X(\emptyset) \in\mathcal{X}$ whose distribution is $\nu$. The initial individual $\emptyset$ produces a random number $N=N(\emptyset)$ of particles of generation $1$, denoted by $1,2,\cdots,N$, with traits determined by
 
$$\mathbb{P}(X(1)\in dx_1, \cdots, X(k)\in dx_k|N=k, X(\emptyset)=x)=p^{(k)}(0)(x, dx_1,\cdots,dx_k),$$
where for each $k, n\in\mathbb{N}$ and $x\in\mathcal{X}$, $p^{(k)}(n)(x,\cdot)$ is a probability measure on ($\mathcal{X}^k, B_{\mathcal{X}^k})$.
%{\bf \emph{What is $B_{\mathcal{X}^k}$ ?}\textbf{I mean $(B_{\mathcal X})^k$}}. 
More generally, each individual  $u=u_1\cdots u_n$ of generation $n$ whose trait is $X(u)$ yields $N(u)$ offspring in  generation $n+1$, denoted by $u1, u2,\cdots,uN(u)$, whose traits are determined by
\begin{eqnarray*}
&&\mathbb{P}(X(u1)\in dx_1, \cdots, X(uk)\in dx_k|N(u)=k, X(u)=x)\\
&&\qquad\qquad \qquad \qquad\qquad \qquad \qquad \qquad \qquad \qquad=p^{(k)}(n)(x, dx_1,\cdots,dx_k).
\end{eqnarray*}
The individuals of each generation evolve independently, so the process enjoys the branching property.

The evolution of a time homogeneous Markov chain indexed by a binary tree is  well known thanks to the works of  \cite{atarbre, guy}. As soon as the Markov chain along a random lineage of the binary tree is ergodic,  a  law of large numbers holds. It yields the convergence for the proportions of individuals in generation $n$ whose trait has some given value.
% (i.e. the asymptotic density occupation of the associated particle system).
 More specifically, this asymptotic proportion  is given by the stationary measure of the ergodic Markov chain. This convergence holds in probability in general, and under additional assumptions on the speed of convergence of the  Markov chain or uniform ergodicity, it also holds almost surely.  Such results  have been extended and modified to understand the (random) transmission of some biological characteristic of dividing cells such as cellular aging, cell damages, parasite infection... In particular, \cite{ban} considered non ergodic Markov chains and rare events associated to the Markov chain for cell division with parasite infection. In \cite{delmar}, the authors considered a Markov chain indexed by a Galton-Watson tree, which is motivated by cellular aging when the cells may die. In the same vein, the almost sure convergence in the case of bifurcating autoregressive Markov chain is achieved in  \cite{bercal}   via  martingale arguments. Such results have been extended recently and one can see the works of De Saporta, G\'egout- Petit, Marsalle, Blandin and al. %In this particular case,  the estimations of the parameters of the autoregressive kernel have been obtained in  \cite{vassili}, with kernel depending on a varying environment and population still described by a  binary tree.  
For biological motivations in this vein, we also refer to \cite{tad, sincl}. 

In this paper, we consider similar questions for the case where  both  the genealogical tree and the  Markov chain along the branches are time non-homogeneous. In particular, we are motivated by the fact that the cell division is affected by the media. This latter is often  time non-homogeneous, which may be due  to the variations of the available resources  or  the environment, a medical treatment... Such  phenomena are well known in biology from the classical studies of Gause about Paramecium  or Tilman about diatoms.  The cell genealogy may  be modeled by a Galton-Watson process in a varying (or random) environment. It is quite straightforward  %As it will appear, 
to extend the weak law of large numbers to the case of non-homogeneous genealogies if the branching events are symmetric and independant  (each child obtains an i.i.d. copy) and the Markov chain along the branches is time homogeneous.  
%It means that the  Markov chain along all the branches of the tree are time homogeneous Markov chains identically distributed. 
However, as the convergence of non-homogeneous Markov chains is a delicate problem, we need  to consider new limit theorems to understand the evolution of the traits in the cell population. As stated in the next section, the asymptotic proportion can still  be  characterized as the stationary probability of an auxiliary Markov chain, in the same vein as \cite{delmar}. It yields a natural interpretation of the repartition of the traits as a stationary probability and the description of the lineage of a typical individual, which then can  be easily simulated.  
%For biological motivations, we refer to ??.\\
%ultra classiques de Gause et ses paramécies, ou de Tilman et de ses diatomées. C'est de l'interspécifique mais tu trouveras ce que tu souhaites là je pense.
%http://books.google.fr/books?id=XzK4kmJws9UC&pg=PA395&lpg=PA395&dq=Tilman+et+de+ses+diatom%C3%A9es.&source=bl&ots=VlRjeUKvhS&sig=rL28s3YXsCe3tQNv7Sfy9LoV6nw&hl=en&redir_esc=y#v=onepage&q=Tilman%20et%20de%20ses%20diatom%C3%A9es.&f=false
A large literature also  exists  concerning asymptotic behavior of even-odd Markov chains along  time homogeneous trees (see e.g. \cite{PhD}), with different motivations. 
We stress that in our model the trait of the cell does not influence its division, which means  that the genealogical tree may be random but does not depend on the evolution of the Markov chain along its branches. When such a dependence  holds (in continuous time, with fixed environments), some many to one formulas can be found in \cite{harrisroberts} and asymptotic proportions were briefly considered in \cite{bantran}. \\

%\marginpar{chunmao, could you improve/complete/change ?}
Letting the trait be (replaced by) the location of the individual, the process considered here is more usually called a \emph{Branching Markov Chain}. The particular case that the motion of each individual has  i.i.d. increments, i.e. \emph{branching random walks}, has been largely studied from the pioneering works of Biggins  \cite{biggins77, biggins772,b}. The density occupation, the law of large numbers, central limit theorems,  large deviations results and the positions of the extremal particles have been considered. %\marginpar{We should try to be more precise.\emph{Ok, we do it later.}} 
These results, such as recurrence, transience or survival criteria, have been partially extended to random environment  both in time and space, see e.g. \cite{gant, muller, huang2, CP, CY, yosh, MN}.  
Here, we consider both the case of non i.i.d. displacements  and non homogeneous environments. We mainly focus   on (positive) recurrent branching random walks. This provides some tractable  models for  reproduction-dispersion of species which evolve in a spatially and temporally non-homogeneous environment and compact state space. With our assumptions, the time environment may influence both the reproduction and the dispersion. 
  One can figure out the effects of the humidity and the enlightenment for the reproduction of plants and the wind for pollination. The spatial environment (such as the intensity of wind, the relief...) may here influence only  the dispersion. The fact that the space location does not influence the reproduction events  requires  space homogeneity of some environmental parameters, such as the light exposure, quality of the ground. An extension of our results to space dependent reproduction is a challenging problem. As a motivating article in this vein in ecology,
 we refer for instance to \cite{frev}. \\

Finally, such models might be a first step to consider  evolution processes on larger time scale with time inhomogeneity. The trait would then correspond to a phenotype or  a genotype. The fact that the branching event does not depend on the trait (neutral theory) may hold in some cases or be used as a zero hypothesis, see for example \cite{LaPo}. More generally, the  non-homogeneity of the branching rates in the genealogies raises many difficulties but has various motivations. As an example,   we refer to \cite{stadler} for discussions on  time non-homogeneity for extinction and speciation. \\

In next two sections, we state the results of this paper and consider some applications. Firstly (Section 2), we give a very general statement  which ensures the convergence in probability of the proportions of individuals with a given trait as time goes to infinity.  The fact that the common ancestor of two individuals is not recent yields a natural setting for law of large numbers. Here the genealogical tree may  be very general but the assumptions that we need are  often  not satisfied and the asymptotic proportions are not explicit.
That's why we focus  next on time non-homogeneous tree with branching properties (Section 3). That allows us to get a many to one formula, in the same vein as \cite{guy, bantran, harrisroberts, krell} (see Lemma \ref{LLL3.2}). We can then state
 a forward law of large numbers, which requires ergodic convergence of an auxiliary  time non-homogeneous Markov chain.  This latter is still not easily  satisfied. Thus, we give a backward analog of this result, which yields quite general sufficient conditions linked to ergodicity of Markov chains with stationary Markov transitions \cite{Orey}. To get assumptions which are easy to check and applications with stationary ergodic environments, we also provide  a weak law of large numbers for the whole population. %This framework will be natural for modelisation in dispersion reproducion models or cell dynamics.  
Finally, we derive a central limit theorem  and apply it to some branching random walks in random environments. The rest of the paper (Sections 4 and 5) is dedicated to the proofs.\\

\paragraph{Notations.}% and definition of the model.}
In the whole paper, we need the following notations. If $u=u_1\cdots u_n$ and $v=v_1\cdots v_m$, then $\vert u\vert=n$ is the length of $u$ and $uv=u_1\cdots u_nv_1\cdots v_m$. We denote by 
$$\mathbb{T} \in \cup_{m=0}^{\infty} \{1,2,\cdots\}^m$$
the generation tree rooted at $\emptyset$ and we define by
 $$\T_n:=\{u\in\T:|u|=n\}$$
 the set of all individuals in generation $n$.  Let
$$Z_n=\sum_{u\in\T_n}\delta_{X(u)}$$
be the counting measure of particles of generation $n$. In fact, for any measurable set $A$ of $\mathcal X$,
$$Z_n(A):=\#\{u\in\T_n:X(u)\in A\}$$
denotes the number of individuals whose trait belongs to $A$. Our aim is to obtain the asymptotic behavior of this quantity.
In particular, we write $$N_n:=Z_n(\mathcal X)$$ and we   shall consider  the asymptotic proportion of individual whose trait belongs to $A$, which is given by $Z_n(A)/N_n$.

For two different individuals $u, v$ of a tree, write $u<v$ if $u$ is an ancestor of $v$, and denote by $u\wedge v$ the nearest common ancestor of $u$ and $v$ in the means that $|w|\leq |u\wedge v|$ if $w< u$ and $w< v$.

%\marginpar{\textbf{Changed}}

\section{Weak law of large numbers for   non-homogeneous trees}
\label{gen}

%\marginpar{changed to proposition since we have too many theorems compared to our proofs\textbf{OK}}
%\marginpar{changed !!}
In this section, the genealogical tree $\T$ is fixed (non random).%, which means that it is not random. 
We require that the size of the population in generation $n$ goes to infinity as $n\rightarrow \infty$.
% goes to infinity.   %\marginpar{Ref about trait independent genealogical tree.}

We consider a transition kernel $(p^{(k)}(n)(x, dx_1,\cdots,dx_k) : k,n \geq 0)$. Then  the Markov chain $X$ along the tree $\T$  is specified recursively by 
\begin{eqnarray}
&&\mathbb E \left[ \prod_{u \in \T_n} F_u\left(X(u1),\cdots,  X(uN(u))\right)    \bigg| (X(u) : \vert u\vert  \leq n)\right] \nonumber \\
&&\qquad \qquad \qquad \qquad=\prod_{u \in \T_n} \int F_u(x_1,\cdots,x_{N(u)}) p^{N(u)}(n)(x,dx_1,\cdots dx_k). \label{constr}
\end{eqnarray}
%\marginpar{to be precised !!}
where $(F_u : u\in \T) \in \mathcal{B}_b^{\T}$ and $ \mathcal B_b$ is the set of bounded  measurable functions from $\cup_{k\geq 0} \mathcal{X}^k$  to $\R$.
% (in finite dimension ??? $\R$ is enough I believe \textbf{Yes, writing $\R$} is enough ) .\\

The trees rooted at $u$ are defined similarly:
$$\T(u):=\{ v : uv\in\T\}, \quad  \T_n(u):=\{v  :  uv\in\T,  \ |v|=n\}$$
 $$Z^{(u)}_n:=\sum_{uv\in\T_n(u)}\delta_{X(uv)}, \quad N_n(u):=Z^{(u)}_n(\mathcal{X}).$$

%\marginpar{ Be careful : convergence to $\mu(A)$ can be checked only when $\mu(\delta A)=0$. \emph{WHY?} That's the classical weak convergence, see e.g. Billingsley, or the course in prepa agreg on my web page (chapter 4).  And consider for example $\delta_{1/n} \Rightarrow \delta_0$ but $\delta_{1/n}\{0\}=0$,$\delta_{0}\{0\}=1$.   }

\begin{pr}\label{LLT1}
Let $A\in B_{\mathcal X}$. We assume that 
\begin{itemize}
\item[(i)] $N_n\rightarrow\infty$ as $n\rightarrow\infty$;
\item[(ii)]$\limsup_{n\rightarrow\infty}\mathbb{P}(|U_n\wedge V_n|\geq K)\rightarrow 0$ as $K\rightarrow\infty$, where $U_n$, $V_n$ are two  individuals uniformly and independently chosen in $\T_n$;
\item[(iii)] there exists $\mu(A)\in\mathbb R$ such that for all $u\in\T$ and $x \in \mathcal X$,
$$\lim_{n\rightarrow\infty}\mathbb{P}\left(X(U_n^{(u)})\in A\bigg|X(u)=x\right)=\mu(A),$$
where $U_n^{(u)}$ denotes an individual uniformly chosen in  $\T_n(u)$.\\
\end{itemize}
%\marginpar{$N_n(u)$ is given by since the tree is fixed. \emph{YES!}}
Then
 $$\frac{Z_n(A)}{N_n}\rightarrow \mu(A)\qquad \text{in $L^2$.}$$
\end{pr}

%\marginpar{changed}

The Assumption $(ii)$  means that  the common ancestor of two individuals chosen randomly is at the be beginning of the tree. Then assumption $(iii)$  ensures that any sampling is giving the same distribution.
 The assumptions $(i-ii)$ hold for many classical genealogies, such as branching genealogies (see below), Wright Fischer (or Moran) genealogies when we let the size of the population $N$ go to infinity.
The assumption $(iii)$ is  difficult to obtain in general. We first give a simple example where it holds. The next section is giving better sufficient conditions, in the branching framework.
 We also  provide below a result which weaken the Assumption $(ii)$, since the most recent common ancestor can be in the middle of the tree. It requires a stronger ergodicity along the branches than Assumption $(iii)$.

\paragraph{Example 2.1. Symmetric independent kernels.} The Assumption $(iii)$ becomes clear in the symmetric and homogeneous case. More precisely it holds if  
$$p^{(k)}(n)(x, dx_1,\cdots,dx_k)=\prod_{i=1}^k p(x,dx_i)$$
and  $\mathbb P_x(Y_n \in A) \rightarrow \mu(A)$ as $n\rightarrow \infty$, for every $x$, where $Y_n$ is a Markov chain with transition kernel $p$. This problem is related to the ergodicity of $Y$. Sufficient conditions for the ergodicity of a Markov chain are known in literature, and we refer e.g. to 
\cite{MT}. \\
%To give applications to more general transition kernels $p^{(k)}(n)$ than in the example above, we focus in the next section on branching genealogical trees.\\

% If this common ancestor is neither at the be beginning nor the end of the tree, we have another result below, which  

\begin{pr}\label{LLT11}
Let $A\in B_{\mathcal X}$. We assume that 
\begin{itemize}
\item[(i)] $N_n\rightarrow\infty$ as $n\rightarrow\infty$;
\item[(ii)]$\limsup_{n\rightarrow\infty}\mathbb{P}(|U_n\wedge V_n|\geq n-K)\rightarrow 0$ as $K\rightarrow\infty$, where $U_n$, $V_n$ are two  individuals uniformly and independently chosen in $\T_n$;
\item[(iii)] there exists $\mu(A)\in\mathbb R$ such that %for all  $x \in \mathcal X$, 
$$\lim_{n\rightarrow\infty}\sup_{u\in\mathbb T,x\in \mathcal X}\left|\mathbb{P}\left(X(U_n^{(u)})\in A\bigg|X(u)=x\right)-\mu(A)\right|=0,$$
where $U_n^{(u)}$ denotes an individual uniformly chosen in  $\T_n(u)$.\\
\end{itemize}
%\marginpar{$N_n(u)$ is given by since the tree is fixed. \emph{YES!}}
Then
 $$\frac{Z_n(A)}{N_n}\rightarrow \mu(A)\qquad \text{in $L^2$.}$$
\end{pr}

%\marginpar{changed}
We note that the assumption $(ii)$  is satisfied  for any tree $\T$ where each individual has at most $q$ (constant) offspring.
Considering the  symmetric and homogeneous case described in Example 2.1, $(iii)$ is satisfied when a strong ergodicity holds. It it the case for example in finite state space or under Doeblin type conditions. In this situation,   Proposition \ref{LLT11} can be applied.
%, it can be applied to any such tree.

\section{Quenched Law of large numbers for branching Markov chains in random environment}
\label{LLNBP}
In this section, the genealogical tree $\T$  may be  random.  The the population evolves following a branching process in  random environment (BPRE), described as follows. Let $\xi=(\xi_0,\xi_1,\cdots)$ be a sequence of random variables taking values in some measurable space $\Omega$, which will come in applications below from a  stationary and ergodic process. Each $\xi_n$ corresponds to a probability distribution on $\mathbb N=\{0,1,2,\cdots\}$, denoted by $p(\xi_n)=\{p_k(\xi_n):k\geq0\}$. This infinite vector  $\xi$ is called a \emph{random environment}. 

% Given the environment $\xi$,

%\marginpar{\textbf{changed}}
We consider now random measurable transition kernels $(p^{(k)}_{\xi_n}(x, dx_1,\cdots,dx_k) : k,n \geq 0)$, which are indexed by the $n$th environment component $\xi_n$.
 The process $X$ is a Markov chain along the random tree $\T$ with transition kernels $p$.  Conditionally on $(\xi,\T)$, the process is constructed following $(\ref{constr})$.
More specifically, the successive offspring  distributions are $\{p(\xi_{n})\}$, so that  the number of offspring $N(u)$ of individual $u$ of generation $n$
is distributed as % , the number of its offspring $N(u)$ is determined by the distribution 
$p(\xi_n)$ and the traits of its offspring $\{X(ui)\}$ are determined by
%\marginpar{I would prefer to change to this notation for environment. But as you wish...}
\begin{eqnarray*}
&&
\mathbb{P}_\xi(X(u1)\in dx_1,\cdots, X(uk)\in dx_k|N(u)=k,X(u)=x)\\
&& \qquad \qquad \qquad \qquad \qquad \qquad \qquad \qquad \qquad \qquad =p^{(k)}_{\xi_n}(x, dx_1,\cdots, dx_k).
\end{eqnarray*}
We note that the offspring number $N(u)$ does not depend on the parent's trait $X(u)$, and the offspring traits $\{X(ui) :i=1,\cdots, N(u)\}$ may depend on   $N(u), X(u)$ and $\xi_n$.

Given $\xi$, the conditional probability will be denoted by
$\mathbb{P}_{\xi}$ and the corresponding expectation by $\mathbb{E}_{\xi}$. The total
probability will be denoted by $\mathbb{P}$ and the corresponding expectation
by $\mathbb{E}$. As usual, $\mathbb{P}_\xi$ is called \emph{quenched law}, and $\mathbb{P}$ \emph{annealed law}.

Let $\mathcal {F}_0=\mathcal
{F}(\xi)=\sigma(\xi_0,\xi_1,\cdots)$ and $\mathcal
{F}_n=\mathcal
{F}_n(\xi)=\sigma(\xi_0,\xi_1,\cdots,(N(u):|u|<n))$ be the $\sigma$-field generated by the random
variables $N(u)$ with $|u|<n$, so that $N_n$, the size of the population in generation $n$,
  are $\mathcal {F}_n$-measurable.
Denote $$m_n=\sum_{k}p_k(\xi_n)\qquad\text{ for } n\geq0,$$
$$P_0=1\qquad\text{and}\qquad P_n=m_0\cdots m_{n-1}\;\;\text{for  } n\geq 1.$$
Thus, for every $n\in \mathbb N$,  $P_n=\mathbb{E}_\xi N_n$. It is well known that the normalized population size $$W_n=\frac{N_n}{P_n}$$ is a nonnegative martingale with respect to $\mathcal
{F}_n$, so  the limit $$W=\lim_{n\rightarrow\infty}W_n$$ exists a.s. and $\mathbb{E}_\xi W\leq1$.

In the rest of this section, we make the following assumptions

%\marginpar{notations changed, ok ?}
\begin{hyp}
%\marginpar{Maybe we could cite these assumptions when they are really used in the proof\textbf{OK, I'll check the proofs later.}}
\begin{itemize}
\item[(i)] The environment   $\xi=(\xi_0,\xi_1,\cdots )$ is a stationary ergodic sequence. 
\item[(ii)] We assume that $\mathbb P(m_0=0)=0$, $\mathbb P(p_0(\xi_0)=1)<1$  and $\mathbb E(\log m_0)<\infty.$
\item[(iii)] We focus on the  supercritical non degenerated case
\begin{equation}
\label{assBPRE}
\mathbb{E}(\log m_0)>0, \qquad  \mathbb{E}\left(\log\frac{\mathbb{E}_\xi N^2}{m_0^2}\right)<\infty.
\end{equation}
\end{itemize}
\end{hyp}

The first assumption allows to get asymptotic results on the size of the population. The second assumption avoids some degenerated cases.
Denoting by  $$q(\xi):=\mathbb P_\xi(N_n=0 \;\text{for some $n$})$$ the extinction probability, it is 
 well known that the non-extinction event $\{N_n \rightarrow\infty\}$ has quenched probability $1-q(\xi)$. Moreover,  
the condition  $\mathbb E(\log m_0)\leq 0$ implies that $q(\xi)=1 \ \text{a.s.},$
whereas $\mathbb E(\log m_0)>0$ (supercritical case) yields
$$q(\xi)<1 \qquad \text{a.s.}$$
The last assumption ensures that the random variable  $W$ is positive on the non-extinction event. We refer to \cite{at, atr} for the statements and proofs of these results. \\
%We are now focusing on this supercritical case  ($\mathbb{E}\log m_0>0$). Under the additional assumption
%$\mathbb{E}\left(\frac{N}{m_0}\log^+N\right)<\infty$, we have from Theorem 1 in \cite{atr} 
%$$\mathbb{P}_\xi(W>0)=\mathbb{P}_\xi(N_n\rightarrow\infty)<1  \qquad \text{a.s.}$$
%Moreover let
%$$\xi^{(n)}=(\xi_n,\cdots \xi_0)$$
%\marginpar{OK ?}
%and note that
%$$\mu_n:=\mathbb P_{\xi^{(n)}}( W_n \in . ) \stackrel{d}{=}\mathbb P_{\xi} (W_n \in .)$$

\subsection{Forward weak law of large numbers in generation $n$}\label{LLNBPF}
%\marginpar{Quenched - we need to  specify for a.e. $\xi$ no ?\emph{ YES} \\
%Moreover we should state the result on the event $\{\forall n : N_n>0\}$ (on the LEFT and right hand side) to avoid division by $0$ no ? }

We first give a forward law of large numbers in generation $n$ for the model introduced above, with the help of an auxiliary Markov process constructed as follows.
Let
$$P_{\xi_n}^{(k,i)}(x,\cdot)=p^{(k)}_{\xi_n}(x, \mathcal{X}^{i-1}\times\cdot\times\mathcal{X}^{k-i}) $$
and the random transition probability
%\marginpar{I think $Q_n$ should be introduced just above since it is used in the different subsections. \emph{Maybe, but it is also not appropriate to appear in last section. We deal with it later.}}
$$Q_n(x,\cdot)=Q(T^n\xi;x,\cdot):=\frac{1}{m_n}\sum_{k=0}^\infty p_k(\xi_n)\sum_{i=1}^k P_{\xi_n}^{(k,i)}(x,\cdot).$$
We note that for each $\xi\in\Omega$, the Markov transition kernel $Q(\xi;\cdot,\cdot)$ is a function from $\mathcal X\times B_{\mathcal X}$ into $[0,1]$ satisfying:
\begin{itemize}
\item for each $x\in\mathcal X$, $Q(\xi;x,\cdot)$ is a probability measure on ($\mathcal X, B_{\mathcal X}$);
\item for each $A\in B_{\mathcal X}$, $Q(\xi;\cdot,A)$ is a $B_{\mathcal X}$-measurable function on $\mathcal X$.
\end{itemize}

Given the environment $\xi$, we define an auxiliary Markov chain in varying environment $Y$, whose 
transition probability in generation $j$ is  $Q_j$ :
%  $ with $Y_0(u)=X(u)$ and transition probability
$$\mathbb{P}_\xi(Y_{j+1}=y|Y_j=x)=Q_j(x,y).$$
%In particular, we write $Y_n=Y_n(\emptyset)$. 
As usual, we denote by $\mathbb P_{\xi,x}$ the quenched probability when the process $Y$ starts  from the initial value $x$,  and by $\mathbb E_{\xi,x}$ the corresponding expectation.

As stated below, the  convergence of the measure $Z_n(\cdot)$  normalized  comes from
the ergodic behavior of $Y_n$. In the same vein as \cite{delmar}, we have
\begin{thm}\label{LLNT3.1'}
Let $A\in B_{\mathcal X}$. We assume that there exists a sequence $(\mu_{\xi,n}(A))_n$  
%\marginpar{no need of $\mu_{n,\xi}(A)\leq 1$, since it is limit of a probability. 
%\emph{No! e.g. $A=\mathbb R$, $\mu_{n,\xi}(\mathbb R)=1+1/n>1$. But Prop 4.1 can be improved.}}
such that for almost every $\xi$ and for each $r\in \mathbb N$,
\begin{equation}\label{LLET3.2b}
\lim_{n\rightarrow\infty}\mathbb{P}_{T^r\xi,x}(Y_{n-r}\in A)-\mu_{\xi,n} (A)=0\qquad \text{for every $x\in\mathcal{X}$,}
\end{equation}
where $T\xi=(\xi_1,\xi_2,\cdots)$ if $\xi=(\xi_0,\xi_1,\cdots)$.
Then we have for almost all $\xi$,
\begin{equation}\label{LLE3.2b}
\frac{Z_n(A)}{P_n}-\mu_{\xi,n}(A)W \rightarrow 0 \qquad \text{in $\mathbb{P}_\xi$-$L^2$,}
\end{equation}
and conditionally on the non-extinction event, 
\begin{equation}\label{LLE3.3b}
 \frac{Z_n(A)}{N_n} - \mu_{\xi,n}(A)\rightarrow 0 \qquad \text{in $\mathbb{P}_\xi$-probability.}
\end{equation}
\end{thm}
 This forward result theorem is adapted to the underlying branching  genealogy. The proof is defered to the next section, where a more general is obtained. The condition (\ref{LLET3.2b}) 
 holds if the auxiliary Markov chain is weakly ergodic, for suitable sets $A$. For sufficient (and necessary) conditions of weak ergodicity in the non-homogeneous case, we refer in particular to \cite{mukh}. 
 
Let us now give more trackable results. We derive 
a first result of (quenched forward) weak law of large numbers, under a stronger assumption, and two  examples in simple cases (homogeneous case).

\begin{co}\label{LLT3.1}
Let $A\in B_{\mathcal X}$. We assume that
there exists $\mu(A) \in \R$ such that for almost all $\xi$,
\begin{equation}\label{LLET3.2}
\lim_{n\rightarrow\infty}\mathbb{P}_{\xi,x}(Y_n\in A)=\mu (A)\qquad \text{ for every $x\in \mathcal X$.}
\end{equation}
Then we have  for almost all $\xi$,
\begin{equation}\label{LLE3.2}
\frac{Z_n(A)}{P_n}\rightarrow \mu(A)W \qquad \text{in $\mathbb{P}_\xi$-$L^2$,}
\end{equation}
and conditionally on the non-extinction event,
\begin{equation}\label{LLE3.3}
\frac{Z_n(A)}{N_n}\rightarrow \mu(A)\qquad \text{in $\mathbb{P}_\xi$-probability.}
\end{equation}
\end{co}

We now give two examples, where we can check (\ref{LLET3.2}) by considering models for which the associated auxiliary chain $Y$ is time homogeneous.
%\marginpar{To be checked\textbf{\\???? haven't checked.}}
We also note  that assuming a uniform convergence in with respect to $x$, we can get an almost sure  convergence following the proof of Theorem $2$ in \cite{atr}.

\paragraph{Example 3.1. Homogeneous Markov chains along Galton Watson  trees.}
We focus here on the case when the time environment is non random, i.e. $\xi_n$ is constant for every $n\in\mathbb N$. The genealogical tree is a Galton Watson tree, whose offspring distribution is specified  by $\{p_k : k\geq  0\}$.
Moreover, we assume that
\begin{eqnarray*}
\mathbb{P}_{\xi}(X(u1)\in dx_1, \cdots, X(uk)\in dx_k|N(u)=k, X(u)=x)
=p^{(k)}(x, dx_1,\cdots,dx_k).
%&&\qquad\qquad \qquad \qquad\qquad \qquad \qquad \qquad=p^{(k)}(x, dx_1,\cdots,dx_k)
\end{eqnarray*}
does not depend on $\xi$. Then, denoting by $m$ the mean number of offspring per individual and
$$P^{(k,i)}(x,\cdot)=p^{(k)}(x, \mathcal{X}^{i-1}\times\cdot \times \mathcal X^{k-i}),$$
the auxiliary process $Y$ is a time homogeneous Markov chain whose transition kernel is given by
$$Q(x,\cdot)=\frac{1}{m}\sum_{k=0}^\infty p_k\sum_{i=1}^k P^{(k,i)}(x,\cdot).$$
Thus a law of large number in probability is obtained as soon as the ergodicity of the Markov chain $Y$ is proved. Indeed,
it ensures that there exists a probability measure $\mu$ such that for each $x\in\mathcal{X}$ and a measurable set $A$ such that $\mu(\partial A$\footnote{$\partial A$ is the boundary of $A$})$=0$, and $Q^n(x,A)=\mathbb P_x(Y_n\in A)\rightarrow\mu(A)$, so that we can apply Theorem \ref{LLT3.1}. 

We recall  that  sufficient conditions for the ergodicity of a Markov chain are known in the literatures, see e.g.   
\cite{MT}. This result is a simple generalization of law of large numbers on the binary tree in \cite{guy} and that on Galton Watson trees with at most two offsprings given in \cite{delmar}.
A continuous time analogous result can be found in  \cite{BDMT}.

\paragraph{Example 3.2. Symmetric homogeneous Markov chains along branching trees in  random environment.} 
%\marginpar{$w$ changed with $p$ and $y$ with $dy$ to be coherent with the rest of the paper\emph{OK}}
%We consider a branching random walk on $\mathbb{R}$ in the environment $\xi$. 
We consider a Branching Markov Chain   on $\mathcal X$ in the stationary and ergodic environment $\xi$.
Given $\xi$, for each $u$ of generation $n$, the number of its offspring $N(u)$ is determined by distribution $p(\xi_n)=\{p_k(\xi_n):k\geq0\}$. The offspring positions $\{X(ui)\}$, independent of each other conditioned on the position of $u$, are determined by
$$\mathbb{P}_\xi(X(ui) \in dy|X(u)=x)=p(x,dy),$$
%\marginpar{ $w$ is not a judicious notation since it is also used for nodds of the tree\\\emph{Here is omega $\omega$}}
where $p(x,\cdot)$ is a probability  on $\mathcal{X}$.  We note that this example is a particular case of  Example 2.1. We can see that
\begin{eqnarray*}
 p_{\xi_n}^{(k)}(x,dx_1,\cdots,dx_k) &=&
%&& \qquad = \mathbb{P}_\xi(X(u1)\in dx_1,\cdots, X(uk)\in dx_k|N(u)=k, X(u)=x)\\
%&&\qquad  =
\prod_{i=1}^k\mathbb{P}_\xi(X(ui)\in dx_i|X(u)=x)=\prod_{i=1}^kp(x,dx_i).
\end{eqnarray*}
Therefore, $Q(x,dy)=Q_n(x,dy)=p(x,dy)$ and $Y_n$ is a  time-homogeneous Markov chain with transition probability $p$.
As in the previous example, the problem is reduced to the ergodicity of $Y$. 
 Such a class may be relevant  to model the dispersion of plants in spaces where the reproduction is homogeneous and the time environment only influences  the genealogy. \\

The convergence of the Markov chain (\ref{LLET3.2}) is difficult to get under general assumptions. Indeed,  soon as the auxiliary kernel $Q$ depends on $\xi$, the auxiliary process is time non-homogeneous and the forward convergence in distribution (\ref{LLET3.2}) won't hold in general\footnote{the reader could consider for example the case of an environment containing only two components whose associated transition matrices have different stationary  probability.}. It is the case  for most of the models we have in mind but the two previous examples. 
For such an ergodic convergence, the backward convergence is much more adapted (see \cite{Orey}). Moreover, general sufficient conditions can be found for ergodic (Birkhoff's) theorems and we can use in particular \cite{Orey,timo}. Thus,  we derive in the next subsection a backward law of large numbers and then  one on the whole tree.

\subsection{Backward law of large numbers in generation $n$}
Now we consider the environment $\xi$ time reversed. Thus, for each $n \in\mathbb N$, we define
$$\xi^{(n)}=(\xi_{n-1}, \cdots, \xi_0).$$
For each individual $u$ of generation $r$ ($0\leq r\leq n$), the number of its offspring $N(u)$ is determined by the distribution $p(\xi_{n-r-1})$, and the positions of its offspring $\{X(ui)\}$ are determined by the transition kernel $p_{\xi_{n-r-1}}$ defined previously. 
%Note that this process stops at generation $n$ since the environment $\xi^{(n)}$ contains $n$ random variables. 
 
To distinguish from the forward case, we denote the counting measure of generation $k$ by $Z_k^{(n)}(\cdot)$, the population size of generation $k$ by $N_k^{(n)}$, and its normalization by $$W_k^{(n)}=\frac{N_k^{(n)}}{\mathbb{E}_{\xi^{(n)}}(N_k^{(n)})} $$ for $0\leq k\leq n$.   We remark that unlike the forward case, here the normalize population $W_n^{(n)}$ is not a  martingale, hence the existence of its limit is not ensured. But when the environment is reversible in law, i.e.
%\marginpar{Be careful, above we start with $\xi_0$, here with $\xi_1$, so we have to homogenize\textbf{OK}}
$$(\xi_0, \cdots, \xi_{n-1})\stackrel{d}{=}(\xi_{n-1}, \cdots , \xi_0),$$
then
$W_n^{(n)}$  has the same distribution as $W_n$ under the total probability $\mathbb P$.

The following theorem is a law of large numbers in generation $n$ for the backward case, whose proof is deferred to next section.

\begin{co}\label{LLT3.2b}
%\marginpar{what is $\mathbb P_\xi$-a.e. x ? Theorem changed to corollary since we had to many theorems}
Let $A\in B_{\mathcal X}$. We assume that  for almost all $\xi$, 
$$m_0\geq a>1,\qquad \mathbb{E}_\xi N^p\leq b$$ for some constants $a,b$ and $p>2$. We also assume  that there exists $\mu_\xi(A) \in \R$ such that   
\begin{equation}\label{LLET32e}
\lim_{n\rightarrow\infty}\mathbb{P}_{\xi^{(n)},x}(Y_n\in A)=\mu_\xi(A)\qquad \text{for every  $x\in\mathcal {X}$.}
\end{equation}
Then  we have  for almost all $\xi$,
\begin{equation}\label{LLNTBE1}
\frac{Z_n^{(n)}(A)-\mu_\xi(A)N_n^{(n)}}{P_n}\rightarrow 0  \;\; \text{in   $\mathbb{P}_{\xi}$-$L^2$.}
\end{equation}
Moreover, if the environment is reversible in law
%\footnote{It means that $(\xi_1, \cdots, \xi_n)$ is distributed as $(\xi_n, \cdots , \xi_1)$. It holds for example if the environments are i.i.d.},
 then
\begin{equation}\label{LLNeT3.2}
\mathbf{1}_{\{N_n^{(n)}>0\}}\left[\frac{Z_n^{(n)}(A)}{N_n^{(n)}}-\mu_\xi(A)\right]\rightarrow 0 \qquad \text{in $\mathbb{P}$-probability.}
\end{equation}
\end{co}

%The proof follows  the previous one, only step last step differs (Ok chunmao ? Details ?), to get
%\begin{equation}\label{LLE4}
%$$\frac{Z_n(A)-\mu(A)N_n}{P_n}\rightarrow 0 \qquad \text{in } \mathbb{P}_{\xi^{(n)}}-L^2.$$
%\end{equation}
%Then
%\begin{eqnarray*}
%&&\mathbb E_{\xi^{(n)}}\left(  \mathbf{1}_{\{N_n>0\}}\left[\frac{Z_n(A)}{N_n}-\mu(A)\right] \right)^2\\
%&&\leq \mathbb E_{\xi^{(n)}}  \mathbf{1}_{\{N_n>0\}}W_n^{-2}.
%\mathbb E_{\xi^{(n)}}\left(\frac{Z_n(A)-\mu(A)N_n}{P_n}\right)^2
%\end{eqnarray*}
%If we find  conditions such that  
%$$\mathbb E_{\xi^{(n)}}\left( \mathbf{1}_{\{N_n>0\}}\frac{1}{W_n^2}\right)<\infty,$$
%(Chunmao, do you know to do that ? it reminds me your thesis. Do we have $\mathbb E(1/W^2)<\infty$ ?)
%then the proof is complete.
 
%\marginpar{To be improved}

Thanks to \cite{Orey} (see in particular Theorem 5.5.), we can check when  the Assumption (\ref{LLET32e}) is satisfied.
It requires that the environment is ergodic,   which holds here. As expected, the author also need the uniqueness of the invariant probability and some irreducibility assumptions. He also requires that some $\delta^*$ is equal to $0$, which is much more technical to get. As a simple case  where such assumptions hold, the author gives the case of finite state space, which yields the following example.

%\textbf{Do we get more general results than Tanny in the case when the offspring distribution does not depend of the type ?\emph{That's not important. The model and results of Tanny are quite standard and perfect, but our model is merely a special example. There is no comparability.}}

\paragraph{Example 3.3.  Multitype branching processes in random environment.} In the case when $X(u)$ belongs to a finite state space $\mathcal X$ for every $u$, the process considered here is a multitype  branching processes in random environment, where the reproduction law of each individual does not depend on its type, but the offspring distribution does. If the kernel $Q$ is irreducible, we get the quenched convergence in probability of the proportion of each type. This asymptotic proportion is identified as the stationary measure of the auxiliary chain $Y$, and can be thus easily simulated. For the growth rate of the whole population in the case when the reproduction may depend on the type, we refer to \cite{tan} for such results under stability assumptions. \\

%\begin{itemize}
%\item $\pi$ ergodic. Ok
%\item $X=M$ : explanations
%\item $\mathcal M_i=\{\Phi\}$ : explanations
%\item $\Lambda(M)=1$ : explanations
%\item $\delta^*=0$ : explanations
%\end{itemize}

To get weaker assumptions (of the Doeblin type) which can be satisfied for Markov chains in stationary random environment, we   focus now on limit theorems  the whole tree.
% to use Cesaro convergence of the Markov chain (i.e. ergodic theorem).

%\paragraph{Example :} cas fini, MBPRE (compare with literature). \emph{To be written}. \\
%The condition of convergence of the Markov chain (...) is still difficult to track (see Orey for discussions).
%LEt us now consider the whole tree to give more natural assumptions on our Markov chain (Doeblin condition).

\subsection{Law of large numbers on the whole tree}
In the framework  of Markov chains with stationary and ergodic environments, quenched ergodic theorems  are known (see e.g. \cite{Orey, timo}). They  ensure the convergence (for every $x\in \mathcal X$) of
$$\frac{1}{n}\sum_{k=1}^{n} Q_0\cdots Q_{k-1}(x,\cdot) \qquad \text{as n} \rightarrow \infty.$$
It leads us to consider the following limit theorems on the whole tree, where each generation of the tree has the same mean weight in the limit. Such an approach is both adapted to the branching (forward) genealogy and the convergence of the underlying auxiliary time non-homogeneous Markov chain $Y$, whose transition are stationary and ergodic. It defers from the usual limit theorem on the whole tree  \cite{guy,delmar} where each cell has the same weight, but not each generation.

\begin{thm}\label{LLT3.3}
Let $A\in B_{\mathcal X}$. We assume that
 there exists  $\mu(A) \in \R$ such that  for almost all $\xi$,
\begin{equation}\label{LLET3.3}
\lim_{n\rightarrow\infty}\frac{1}{n} \sum_{k=1}^n\mathbb{P}_{\xi,x}(Y_k\in A)=\mu (A)\qquad \text{for every $x\in\mathcal{X}$.} 
\end{equation}
Then we have for almost all $\xi$,
\begin{equation}\label{LLNEW1}
\lim_{n\rightarrow\infty}\frac{1}{n}\sum_{k=1}^{n}\frac{Z_k(A)}{P_k}=\mu(A)W\qquad \text{in $\mathbb{P}_\xi$-$L^2$.}
\end{equation}
%\marginpar{Do you agree with this additional result ?\textbf{Yes,but in $\mathbb{P}_\xi$. Proof is added after. Example 3.4 is added.}}
and, on the non extinction event,
\begin{equation}\label{LLNEW2}
\lim_{n\rightarrow\infty}\frac{1}{n}\sum_{k=1}^{n}\frac{Z_k(A)}{N_k}=\mu(A)\qquad \text{in $\mathbb{P}_\xi$-probability.}
\end{equation}
\end{thm}
 
 The proof is also deferred to  next section.

\paragraph{Example 3.4. Doeblin conditions for the auxiliary Markov chain $Y$.} Assume that there exist a positive integer $b$ and a measurable function $M(\xi):\Omega \rightarrow [1,\infty)$ such that $\mathbb E|\log M(\xi)|<\infty$, and for almost all $\xi$,
$$\mathbb{P}_{\xi,x}(Y_b\in A)\leq M(\xi)\mathbb{P}_{\xi,y}(Y_b\in A)\qquad \text{for all $x,y\in\mathcal{X}$,}$$
which means that Assumption (A) of Sepp\"al\"ainen \cite{timo} is satisfied. According to  Theorem 2.8 $(i)$ and $(iii)$ of Sepp\"al\"ainen \cite{timo} (with $f=\mathbf 1_{A\times \Omega}$), there exists a probability $\Phi$ on $\mathcal X \times \Omega$ such that for almost all $\xi$,
$$\lim_{n\rightarrow\infty}\frac{1}{n}\sum_{k=1}^{n}\mathbf 1_A(Y_k)=\Phi(A\times \Omega)\qquad\text{$\mathbb P_{\xi,x}$-a.s. for every $x\in\mathcal X$. }$$
By the dominate convergence theorem, we have for almost all $\xi$,
$$\lim_{n\rightarrow\infty}\frac{1}{n}\sum_{k=1}^{n}\mathbf \mathbb \mathbb P_{\xi,x}(Y_k\in A)=\Phi(A\times \Omega)\qquad\text{ for every $x\in\mathcal X$. }$$
Thus (\ref{LLET3.3}) holds with $\mu(A)=\Phi(A\times \Omega)$, so we can use Theorem \ref{LLT3.3} to get (\ref{LLNEW1}) and (\ref{LLNEW2}).

\subsection{Central limit theorem}
%\marginpar{ref to precise}
When the auxiliary Markov chain $Y$ is a classical random walk on  $\mathcal X\subset \mathbb R$, we know that $Y_n$ satisfies a central limit theorem. Such results have been extended to the framework of random walk in random environment (see e.g. \cite{ali}) and some more general Markov chains (see e.g. \cite{gall}).
It leads us to state the convergence of proportions in the case when $Y_n$ satisfies a central limit theorem.
% \marginpar{According to me, we dot not need that $\Phi$ is gaussian.\textbf{We only need the continuity.}}
\begin{thm}\label{LLT3.5}Let  $\mathcal X\subset \mathbb R$.
We assume that  for almost all $\xi$, $Y_n$ satisfies a central limit theorem: there exits a sequence of random variables $\{(a_n(\xi), b_n(\xi)\}$ satisfying $b_n(\xi)>0$ such that
\begin{equation}\label{LLET3.5}
\lim_{n\rightarrow\infty}\mathbb{P}_{\xi,x}\left(\frac{Y_n-a_n(\xi)}{b_n(\xi)}\leq y\right)=\Phi(y)  \qquad \text{for every $x\in\mathcal{X}$,}
\end{equation}
where  $\Phi$ is a continuous function on $\mathbb R$.
If for each $r\in\mathbb N$ fixed,
\begin{equation}\label{LLCLTab}
\lim_{n\rightarrow\infty}\frac{b_n(\xi)}{b_{n-r}(T^r\xi)}=1\qquad\text{and}\qquad\lim_{n\rightarrow\infty}\frac{a_n(\xi)-a_{n-r}(T^r\xi)}{b_{n-r}(T^r\xi)}=0\quad a.s.,
\end{equation}
then we have for almost all $\xi$, 
\begin{equation}\label{LLE3.8}
\frac{Z_n(-\infty,b_n(\xi)y+a_n(\xi)]}{P_n}\rightarrow \Phi(y)W \qquad \text{in $\mathbb{P}_\xi$-$L^2$,}
\end{equation}
and conditionally on the non-extinction event, 
\begin{equation}\label{LLE3.9}
\frac{Z_n(-\infty,b_n(\xi)y+a_n(\xi)]}{N_n}\rightarrow \Phi(y)  \qquad \text{in $\mathbb{P}_\xi$-probability.}
\end{equation}
\end{thm}

\paragraph{Example 3.5. Branching random walk on $\mathbb R$  with random environment in time.}
This model is considered in Huang \& Liu \cite{huang2}. The environment $\xi=(\xi_n)_{n\in\mathbb N}$ is a stationary and ergodic process indexed by time $n\in\mathbb N$. Each realization of $\xi_n$ corresponds to a distribution $\eta_n=\eta_{\xi_n}$ on $\mathbb N\otimes \mathbb R^{\mathbb N}$. Given the environment $\xi$, the process is formed as follows: at time $n$, each particle $u$ of generation $n$, located at $X(u)\in\mathbb R$,  is replaced by $N(u)$ new particles of generation $n+1$ which scattered on $\mathbb R$ with positions determined by $X(ui)=X(u)+Li(u)$, where the point process $(N(u);L_1(u), L_2(u),\cdots)$ has distribution $\eta_n$. To fit with the notations of this paper, we can see that 
$$p_k(\xi_n)=\eta_n(k,\mathbb R\times \mathbb R\times \cdots),$$
%$$p_{\xi_n}^{(k)}(x,x_1, \cdots, x_k)=\eta_n(k,x-x_1,\cdots, x-x_k),$$
$$p_{\xi_n}^{(k,i)}(x,y)=\eta_n(k,\mathbb{R}^{i-1}\times\{y-x\} \times \R^{k-i})=:q_{\xi_n}^{(k,i)}(y-x),$$
$$Q_n(x,y)=\frac{1}{m_n}\sum_{k=0}^\infty p_k(\xi_n)\sum_{i=1}^kq_{\xi_n}^{(k,i)}(y-x)=:q_n(y-x).$$
We note  that for any measurable function $f$ on $\mathbb R$, 
$$\int f(t) q_n(dt)=\frac{1}{m_n}\mathbb E_\xi \sum_{i=1}^{N(u)}f(L_i(u))\quad (u\in\T_n).$$
Hence $q_n$ is the normalized intensity measure of the point process $(N(u);L_1(u), L_2(u),\cdots)$ for $u\in\T_n$. We define
$$Y_n=\zeta_0+\zeta_1+\cdots+\zeta_n,$$
where $\zeta_j$ is independent of each other under $\mathbb P_\xi$ and the distribution of $\zeta_j$ for  $j\geq1$ is $q_j$. Then $Y_n$ is a non-homogeneous Markov chain, whose transition kernel  satisfies
$$\mathbb P_\xi(Y_{n+1}=y|Y_n=x)=q_n(y-x)=Q_n(x,y).$$ 
%\marginpar{What are the bounds of the integrals $\int$\textbf{\\The whole space $\R$}}
Let $\mu_n=\int_{\R} t q_n(dt)$ and  $\sigma_n^2=\int_{\R} (t-\mu_n)q_n(dt)$. If $|\mu_0|<\infty$ a.s. and $\mathbb E (\sigma_0^2)\in(0,\infty)$, according to Huang \& Liu \cite{huang2}, the sequence $(q_n)$ satisfies a central limit theorem:
$$q_1\ast\cdots\ast q_n(b_n(\xi)y+a_n(\xi))\rightarrow \Phi(y)\qquad a.s.,$$
where 
$$a_n(\xi)=\sum_{i=0}^{n-1} \mu_n, \qquad b_n(\xi)=\left(\sum_{i=0}^{n-1} \sigma^2_n\right)^{1/2}$$
 and $\Phi$ is the distribution function of the standard normal distribution. It follows 
that    (\ref{LLET3.5}) holds for almost all $\xi$. 
Moreover, by the ergodic theorem,  (\ref{LLCLTab}) can be verified. Thus we can apply  Theorem \ref{LLT3.5} to this model and obtain (\ref{LLE3.9}) under  the hypothesis given above. This result  can also be deduced from \cite{huang2}, where the almost sure convergence of (\ref{LLE3.9}) is shown though some tedious calculations. 
%But here as an application of Theorem \ref{LLT3.5}, we give a fast proof for its convergence in quenched probability.

\paragraph{Example 3.6. Branching random walk on $\mathbb Z$ with random environment in time and in locations. } 
%\marginpar{$w$ is now $p$.\emph{No, this is the notation in \cite{liu07}.}}
This model is considered in Liu \cite{liu07}. Let $\xi=(\xi_n)_{n\in\mathbb N}$ be a stationary and ergodic process denoting the environment in time, and $\omega=(\omega_x)_{x\in\mathbb Z}$, which denotes the environment in locations, be another stationary and ergodic process taking values in $[0,1]$. The two sequences $\xi$, $\omega$ are supposed to be independent of each other. Given the environment $(\xi,\omega)$, each $u$ of generation $n$, located at $X(u)\in\mathbb Z$, is replaced at time $n+1$ by $k$ new particles with probability $p_k(\xi_n)$, which move immediately and independently to $x+1$ with probability $\omega_x$ and to $x-1$ with probability $1-\omega_x$. Namely, the position of $ui$ is determined by
\begin{equation*}
\mathbb P_{(\xi,\omega)}(X(ui)=y|X(u)=x)=Q(x,y):=\left\{
\begin{array}{cc}
\omega_x&  \text{if $y=x+1$;}\\
1-\omega_x & \text{if $y=x-1$,}	\\
\end{array}\right.
\end{equation*}
where  $\mathbb P_{(\xi,\omega)}$ denotes the conditional probability given the environment $(\xi, \omega)$.
Notice that when the environment in locations $\omega$ is fixed, this process is the just one considered in Example 3.2 with the state space $\mathcal X=\mathbb Z$ and $p(x,y) =Q(x,y)$. So the transition probability of the Markov chain $Y_n$ is $Q$, which only depends on the environment in locations $\omega$ and is independent of the environment in time $\xi$. We can regard $Y_n$ as a random walk on $\mathbb Z$ in random environment which is studied in Alili \cite{ali}.  By Theorem 6.3 of  Alili \cite{ali} and the continuity of $\Phi$, 
%\marginpar{Quelles hypothèses ?\emph{Too many and too complex. So I advice to see \cite{ali} for details.}} 
under some hypothesis, we have for every $\omega$,
\begin{equation*}  \lim_{n\rightarrow\infty}\mathbb{P}_{\omega,x}\left(\frac{Y_n-n\gamma}{\sqrt{n}}\leq y\right)=\Phi(y)  \qquad \text{for every $x\in\mathbb{Z}$,}
\end{equation*}
where $\Phi$ is the distribution function of the normal distribution $\mathcal N(0,D)$, and $\gamma,D$ are two explicit constants (see \cite{ali} for details). Therefore, we can apply 
Theorem \ref{LLT3.5} and obtain (\ref{LLE3.8}) and (\ref{LLE3.9}) under the probability $\mathbb P_{(\xi, \omega)}$.

\section{Proof of Propositions \ref{LLT1} and \ref{LLT11}}

Proposition \ref{LLT1} is a consequence of the following result with $f(x)=\mathbf{1}_{A}(x)-\mu(A)$. It is extended hereafter  to some class of unbounded functions $f$.

\begin{lem}\label{LLNP1}
Let $f$ be a bounded measurable function on $\mathcal X$. We assume that
\begin{itemize}
\item[(i)]  $N_n\rightarrow\infty$ as $n\rightarrow\infty$;
\item[(ii)]$\limsup_{n\rightarrow\infty}\mathbb{P}(|U_n\wedge V_n|\geq K)\rightarrow 0$ as $K\rightarrow\infty$;
\item[(iii)]for all $u\in\T$ and $x \in \mathcal X$,
$$\lim_{n\rightarrow\infty}R_u(n,x)=0 ,$$
where
$$ R_u(n,x)
%=\frac{\mathbb{E}\left[\sum_{uv\in\T_n(u)}f(X(uv)) |X(u)=x\right]}{N(u)}
 =\mathbb{E}\left[\left.f(X(U_n^{(u)}))\right|X(u)=x\right].$$
\end{itemize}
Then
 $$\frac{\sum_{u\in\T_n}f(X(u))}{N_n}\rightarrow 0\qquad \text{in $L^2$.}$$
\end{lem}

%\underline{Q2 : Could we relax $f$ bounded as for Guyon, delmas marsalle  and your following proof....}\\
%\emph{Yes, but the statements is too ugly. I write it as Prop2.3. What do you think about?}\\

\begin{proof} We first notice that
\begin{equation*}
%\label{LLE2.1}
\mathbb{E}\left(\frac{\sum_{u\in\T_n}f(X(u))}{N_n}\right)^2=\frac{1}{N_n^2}\mathbb{E}\left[\sum_{u\in\T_n}f^2(X(u))\right]+ \frac{1}{N_n^2}\mathbb{E}\bigg[\sum_{\substack{u,v\in\T_n\\u\neq v}}f(X(u))f(X(v))\bigg].
\end{equation*}
We need to prove that both terms in the right side  go to $0$ as $n\rightarrow\infty$. For the first term, since $f$ is bounded, there exists a constant $C$ such that $|f|\leq C$. 
By (i),
$$\frac{1}{N_n^2}\mathbb{E}\left[\sum_{u\in\T_n}f^2(X(u))\right]= \frac{\mathbb Ef^2(U_n)}{N_n}\leq \frac{C^2}{N_n}\rightarrow0\quad \text{as } n\rightarrow\infty.$$
The second term can be decomposed as
\begin{eqnarray*}
&&\frac{1}{N_n^2}\mathbb{E}\left[\sum_{\substack{u,v\in\T_n\\u\neq v}}f(X(u))f(X(v))\right]\\
&&\qquad = \frac{1}{N_n^2}\sum_{r=0}^{n-1}\mathbb{E}\left[\sum_{w\in\T_r}\sum_{\substack{wi,wj\in\T_1(w)\\i\neq j}}\sum_{\substack{wi\tilde{u}\in\T_{n-r-1}(wi)\\wj\tilde{v}\in\T_{n-r-1}(wj)}} f(X(wi\tilde{u}))f(X(wj\tilde{v}))\right]\\
&&\qquad = \sum_{r=0}^{K}\mathbb{E}A_{n,r}+\sum_{r=K+1}^{n-1}\mathbb{E}A_{n,r},
\end{eqnarray*}
where $K$ is a fixed integer suitable large, and
$$A_{n,r}=\frac{1}{N_n^2}\sum_{w\in\T_r}\sum_{\substack{wi,wj\in\T_1(w)\\i\neq j}}\sum_{\substack{wi\tilde{u}\in\T_{n-r-1}(wi)\\wj\tilde{v}\in\T_{n-r-1}(wj)}} f(X(wi\tilde{u}))f(X(wj\tilde{v})). $$
It is clear that
\begin{eqnarray*} 
\mathbb{E}A_{n,r}&=&\frac{1}{N_n^2}\sum_{w\in\T_r}\sum_{\substack{wi,wj\in\T_1(w)\\i\neq j}}\sum_{\substack{wi\tilde{u}\in\T_{n-r-1}(wi)\\wj\tilde{v}\in\T_{n-r-1}(wj)}} \mathbb{E}f(X(wi\tilde{u}))f(X(wj\tilde{v}))\\
&=&  \sum_{w\in\T_r}\sum_{\substack{wi,wj\in\T_1(w)\\i\neq j}}a_{n,r}(wi,wj)R_{n,r}(wi,wj),
\end{eqnarray*} 
where 
$$a_{n,r}(wi,wj)=\frac{N_{n-r-1}(wi)N_{n-r-1}(wj)}{N_n^2}$$
and
$$R_{n,r}(wi,wj)=\mathbb{E}\left[R_{wi}(n-r-1,X(wi)) R_{wj}(n-r-1,X(wj))\right].$$
As $|R_u(n,x)|\leq C$ for all $u$,$n$ and $x$,
\begin{eqnarray*}
\sum_{r=K+1}^{n-1}\mathbb{E}A_{n,r}&\leq&C^2\sum_{r=K+1}^{n-1}  \sum_{w\in\T_r}\sum_{\substack{wi,wj\in\T_1(w)\\i\neq j}} a_{n,r}(wi,wj) \\
&=&C^2\sum_{r=K+1}^{n-1}\mathbb{P}(|U_n\wedge V_n|=r)\\
&\leq& C^2 \mathbb{P}(|U_n\wedge V_n|\geq K+1).
\end{eqnarray*}
By (ii), $\limsup_{n\rightarrow\infty}\mathbb{P}(|U_n\wedge V_n|\geq K+1)\rightarrow 0$ as $K\rightarrow\infty$. Thus $\limsup_
{n\rightarrow\infty}\sum_{r=K+1}^{n-1}\mathbb{E}A_{n,r}$ is negligible for $K$ large enough.
For $0\leq r\leq K$, the fact that $R_u(n-r-1,x)$ goes to zero for a.e. $x$ and is bounded by $C$ with respect to $x$ enables us to apply the dominate convergence theorem and  get
$$ R_{n,r}(wi,wj)\rightarrow 0\quad\text{as $n\rightarrow\infty.$} $$ 
Adding that  $a_{n,r}(wi,wj)$ is bounded by $1$ yields
$$\sum_{r=0}^{K}\mathbb{E}A_{n,r}\rightarrow 0\quad\text{as $n\rightarrow\infty.$}$$ 
This completes the proof.
\end{proof}

We give here an extension of the previous result, to get asymptotic results on unbounded functions  (such as $f(x)=x^{\alpha})$.
\begin{lem} \label{LLN21E}
Let $f$ be a measurable function on $\mathcal X$. We assume that
\begin{itemize}
\item[(i)]   $\mathbb{E}f^2(U_n)/N_n \rightarrow 0$ as $n\rightarrow\infty$;
\item[(ii)]there exists a function $g$ such that  for all $n$, $u\in\T$ and  $x \in \mathcal X$, $|R_u(n,x)|\leq g(x)$;
\item[(iii)] $\mathbb{E}g(X(ui)g(X(uj)))\leq\beta_{|u|}$ for $ui,uj\in\T_1(u)$ and $i\neq j$, and $$\limsup_{K\rightarrow\infty}\limsup_{n\rightarrow\infty}C_{n,K}=0,$$ where $C_{n,K}=\sum_{r=K}^{n-1}\mathbb{P}(|U_n\wedge V_n|=r)\beta_r$;
\item[(iv)]for all $u\in\T$ and  $x \in \mathcal X$,  $\lim_{n\rightarrow\infty}R_u(n,x)=0$. 
\end{itemize}
Then
 $$\frac{\sum_{u\in\T_n}f(X(u))}{N_n}\rightarrow 0\qquad \text{in $L^2$.}$$
\end{lem}

%\marginpar{A fast idea of the proof as an adaptation of the previous one would be good}
\begin{proof}According to the proof of Proposition \ref{LLNP1}, here we only need to show %the negligibility of 
that $\limsup_n\sum_{r=K+1}^{n-1}\mathbb{E}A_{n,r}\rightarrow 0$  as $K\rightarrow \infty$. By (ii) and (iii), for $w\in \T_r$, $$|R_{n,r}(wi,wj)|\leq \mathbb{E}g(X(wi)g(X(wj)))\leq\beta_{r},$$
so that
 \begin{eqnarray*}
 \sum_{r=K+1}^{n-1}\mathbb{E}A_{n,r}\leq \sum_{r=K+1}^{n-1}  \sum_{w\in\T_r}\sum_{\substack{wi,wj\in\T_1(w)\\i\neq j}} a_{n,r}(wi,wj)\beta_r=C_{n,K}.
\end{eqnarray*}
Letting  successively  $n$ and $K$ go to $\infty$ yields the result.
\end{proof}

Proposition \ref{LLT11} is a result of Lemma \ref{LLNP11} below, with $f(x)=\mathbf{1}_{A}(x)-\mu(A)$,
which also can be extended to a result similar to Lemma \ref{LLN21E} for  unbounded functions $f$, but here we omit to state it for technical convenience.

\begin{lem}\label{LLNP11}
Let $f$ be a bounded measurable function on $\mathcal X$. We assume that
\begin{itemize}
\item[(i)]  $N_n\rightarrow\infty$ as $n\rightarrow\infty$;
\item[(ii)]$\limsup_{n\rightarrow\infty}\mathbb{P}(|U_n\wedge V_n|\geq n-K)\rightarrow 0$ as $K\rightarrow\infty$;
\item[(iii)]
$\lim_{n\rightarrow\infty}\sup_{u\in \T}|R_u(n,X(u))|=0.$
\end{itemize}
Then
 $$\frac{\sum_{u\in\T_n}f(X(u))}{N_n}\rightarrow 0\qquad \text{in $L^2$.}$$
\end{lem}

\begin{proof}Similar to the proof of Lemma \ref{LLNP1}, but we split 
$$\sum_{r=0}^{n-1}\mathbb{E}A_{n,r}=\sum_{r=0}^{n-K-1}\mathbb{E}A_{n,r}+\sum_{r=n-K}^{n-1}\mathbb{E}A_{n,r},$$
and show the negligibility of the two terms respectively. By (ii), we fist have
$$\sum_{r=n-K}^{n-1}\mathbb{E}A_{n,r}\leq C^2\limsup_{n\rightarrow\infty}\mathbb{P}(|U_n\wedge V_n|\geq n-K)\rightarrow0\;\;\text{as $K\rightarrow\infty$.}$$
Now we consider $\sum_{r=0}^{n-K-1}\mathbb{E}A_{n,r}$. For $r\leq n-K-1$ and $wi,wj \in\T_1(w)$ ($w\in\T_r$, $i\neq j$),
\begin{eqnarray*}
\left|R_{n,r}(wi,wj)\right|&\leq& \mathbb{E}\left|R_{wi}(n-r-1,X(wi)) R_{wj}(n-r-1,X(wj))\right|\\
&\leq&\sup_{k\geq K}\mathbb{E}\left(\sup_{u\in\T}\left|R_u(k,X(u))\right|^2\right).
\end{eqnarray*}
It follows that 
\begin{eqnarray*}
  \sum_{r=0}^{n-K-1}\mathbb{E}A_{n,r}&\leq&    \sum_{r=0}^{n-K-1} \sum_{w\in\T_r}\sum_{\substack{wi,wj\in\T_1(w)\\i\neq j}} a_{n,r}(wi,wj)\sup_{k\geq K}\mathbb{E}\left(\sup_{u\in\T}\left|R_u(k,X(u))\right|^2\right)\\
  &=&  \mathbb{P}(|U_n\wedge V_n|\leq n-K-1)\sup_{k\geq K}\mathbb{E}\left(\sup_{u\in\T}\left|R_u(k,X(u))\right|^2\right)\\
  &\leq&\sup_{k\geq K}\mathbb{E}\left(\sup_{u\in\T}\left|R_u(k,X(u))\right|^2\right).
\end{eqnarray*}
By (iii) and the dominate convergence theorem, we have
$$\lim_{n\rightarrow\infty}\mathbb{E}\left(\sup_{u\in\T}\left|R_u(n,X(u))\right|^2\right)=0.$$
Thus $ \sup_n\sum_{r=0}^{n-K-1}\mathbb{E}A_{n,r}\rightarrow 0$ as $K\rightarrow \infty$. 
\end{proof}

\section{Proofs for branching Markov chains in random environments}
In this section, we focus on the process in random environments and present proofs for the theorems stated in Section \ref{LLNBP}. 

\subsection{Many to one formula}

%\marginpar{I think the lemma should be stated and proved with $\xi(n)$ instead of $\xi$ and $f_{\xi(n)}$ instead of $f_{\xi,n}$ so that it could be applied directly to the forward and backward LLN}
\begin{lem}\label{LLL3.2}
For each $x\in\R$, $u\in\mathbb T$ and any measurable function $f_{\xi,n}$ on $\mathcal{X}$,
\begin{equation}\label{LLNEMTO}
\mathbb{E}_{\xi}\left[f_{\xi,n}(Y_n(u)) \ \vert \ Y_0(u)=x\right]=\frac{\mathbb{E}_\xi\left[\sum_{uv\in\T_n(u)}f_{\xi,n}(X(uv))\bigg|X(u)=x\right]}{\mathbb{E}_\xi N(u)}.
%=:\frac{\mathbb{E}_{\xi,x} \sum_{u\in\T_n}f(X(u))}{\mathbb{E}_\xi N_n}.
\end{equation}
\end{lem}
\begin{proof}
By the definition of $\{Y_n(u)\}$,
it is easy to see that
\begin{eqnarray*}
&&\mathbb{E}_{\xi}\left[f_{\xi,n}(Y_n(u))|Y_0(u)=x\right]\\
 && \qquad = \Q_{|u|}\cdots \Q_{|u|+n-1}f_{\xi,n}(x)\\
&& \qquad  = (m_{|u|}\cdots m_{|u|+n-1})^{-1}\times \\ &&\qquad \qquad 
\sum_{k_0,\cdots,k_{n-1}=0}^\infty 
\sum_{i_0=1}^{k_0}\cdots\sum_{i_{n-1}=1}^{k_{n-1}} 
p_{k_0}(\xi_{|u|})\cdots p_{k_{n-1}}(\xi_{|u|+n-1})P_{\xi_0}^{(k_0,i_0)}\cdots P_{\xi_{n-1}}^{(k_{n-1},i_{n-1})}f_{\xi,n}(x).
\end{eqnarray*}
On the other hand, we notice that
\begin{eqnarray}
&&\mathbb{E}_{\xi,x}\left[\sum_{ui\in\T_1(u)}f_{\xi,n}(X(ui))\right]\\
&& \quad =
\mathbb{E}_\xi\bigg[\sum_{ui\in\T_1(u)}f_{\xi,n}(X(ui))\bigg|X(u)=x\bigg]\nonumber\\
&& \quad =\sum_{k=0}^{\infty}\mathbb{E}_\xi\bigg[\sum_{i=1}^kf_{\xi,n}(X(ui))\bigg|N(u)=k,X(u)=x\bigg]\mathbb{P}_\xi(N(u)=k)\nonumber\\
&& \quad =\sum_{k=0}^{\infty}p_k(\xi_{|u|})\sum_{i=1}^kP_{\xi_{|u|}}^{(k,i)}f_{\xi,n}(x).
\end{eqnarray}
Thus,

\begin{eqnarray*}
&&\mathbb{E}_{\xi,x}\left[\sum_{uv\in\T_n}f_{\xi,n}(X(uv))  \right]\\
&& \quad =\mathbb{E}_{\xi,x} \left[\sum_{uv\in\T_{n-1}(u)}\sum_{uvi\in\T_1(uv)}f_{\xi,n}(X(uvi)) \right] \\
&& \quad = \mathbb{E}_{\xi,x}  \left[\sum_{uv\in\T_{n-1}(u)}\mathbb{E}_\xi\left[ \sum_{uvi\in\T_1(uv)}f_{\xi,n}(X(uvi))\bigg|\mathcal{F}_{n-1}, X(uv)\right]\right] \\
&& \quad =\mathbb{E}_{\xi,x}  \left[\sum_{uv\in\T_{n-1}(u)}\mathbb{E}_{\xi,X(uv)}\left[\sum_{uvi\in\T_1(uv)}f_{\xi,n}(X(uvi))\right] \right] \\
&& \quad = \sum_{k=0}^\infty\sum_{i=1}^{k}p_{k}(\xi_{|u|+n-1})\mathbb{E}_{\xi,x}  \left[\sum_{uv\in\T_{n-1}(u)}P_{\xi_{|u|+n-1}}^{(k,i)}f_{\xi,n} (X(uv))\right].
\end{eqnarray*}
By iteration, we obtain
\begin{eqnarray*} \mathbb{E}_{\xi,x}\left[\sum_{uv\in\T_n}f_{\xi,n}(X(uv))\right]&=&\sum_{k_0,\cdots,k_{n-1}=0}^\infty\sum_{i_0=1}^{k_0}\cdots\sum_{i_{n-1}=1}^{k_{n-1}}p_{k_0}(\xi_{|u|})\cdots p_{k_{n-1}}(\xi_{|u|+n-1})\\
 &&\qquad \quad \times \  P_{\xi_{|u|}}^{(k_0,i_0)}\cdots \p_{\xi_{|u|+n-1}}^{(k_{|u|+n-1},i_{n-1})}f_{\xi,n}(x),
\end{eqnarray*}
so that
$$\frac{\mathbb{E}_{\xi,x}\left[ \sum_{uv\in\T_n(u)}f_{\xi,n}(X(uv))\right]}{\mathbb{E}_\xi N_n(u)}=\mathbb{E}_{\xi,x}\left[f_{\xi,n}(Y_n(u))\right].$$
It ends up the proof.
\end{proof}

\subsection{Proof of the law of large numbers in generation $n$}
Following the definition of  $Y$  in Section \ref{LLNBPF},   for each $u\in\T$,  we define  the Markov chain $Y(u)$ associated to $u$  by
$$\mathbb{P}_\xi(Y_{j+1}(u)=y|Y_j(u)=x)=Q(T^{|u|+j}\xi;x,y).$$
In particular, $Y_n=Y_n(\emptyset)$. Moreover, for any measurable function $f$ on $\mathcal {X}$, we denote throughout 
$$P_{\xi_n}^{(k,i)}f(x):=\int f(y)P_{\xi_n}^{(k,i)}(x, dy)\qquad \text{and}\qquad Q_jf(x):=\int f(y)Q_j(x, dy).$$

First, Lemma \ref{LLL3.2} below reveals the mean  relation between $Y_n(u)$ and the tree $\T(u)$, in the same vein as \cite{guy,delmar,harrisroberts}.

\label{proofLLN1}
\begin{pr}\label{LLP3.3}
Let $\nu_\xi$ be the distribution of $X(\emptyset)$. We assume that  for almost all $\xi$, there exist  a function $g$, an integer $n_0=n_0(\xi)$  and non negative numbers $(\alpha_n,\beta_n)$$=(\alpha_n(\xi),\beta_n(\xi))$ such that
\begin{itemize}
\item[(H1)] for all $n\geq n_0$ and  $x\in\mathcal{X}$, 
$$\sup_{0\leq r<n}Q_r\cdots Q_{n-1}|f_{\xi,n}|(x)\leq g(x).$$ 
\item[(H2)] for every  $n\geq n_0$,
$$\nu_\xi Q_0\cdots Q_{n-1}f_{\xi,n}^2\leq \alpha_n, \qquad \alpha_n/P_n\rightarrow0  \ \ (n\rightarrow\infty);$$
\item[(H3)] for every $n\in\mathbb N$, 
$$\nu_\xi Q_0\cdots Q_{n-1}J_{T^n\xi}(g\otimes g)\leq \beta_n, \qquad
\sum_n\frac{\beta_n}{{P_n}m_n^2}<\infty,$$
 where $J_\xi(f_\xi\otimes g_\xi)(x):=\mathbb{E}_{\xi,x}\sum_{\substack{i,j\in \T_1\\i\neq j}}f_\xi(X(i))g_\xi(X(j))$;
\item[(H4)] for each  $r$ fixed, $Q_r\cdots Q_{n-1}f_{\xi,n}(x)\rightarrow0$ as $n\rightarrow\infty$ for every $x\in\mathcal {X}$.
\end{itemize}
Then we have for almost all $\xi$,
$$\frac{\sum_{u\in \T_n} f_{\xi,n}(X(u))}{P_n}\rightarrow0\qquad \text{in $\mathbb{P}_\xi$-$L^2$.}$$
\end{pr}

\begin{proof}
We need to show that
\begin{equation*}\label{LLE3.5}
\mathbb{E}_\xi\left(\frac{\sum_{u\in\T_n}f_{\xi,n}(X(u))}{P_n}\right)^2\rightarrow0\qquad a.s.\quad\text{as $n\rightarrow\infty$.}
\end{equation*}
Similar to the proof of Proposition \ref{LLNP1}, we write
\begin{eqnarray}\label{LLNEsp}
\mathbb{E}_\xi\left(\frac{\sum_{u\in\T_n}f_{\xi,n}(X(u))}{P_n}\right)^2
&=&\frac{1}{P_n^2}\mathbb{E}_\xi\left[\sum_{u\in\T_n}f_{\xi,n}^2(X(u))\right]\nonumber\\
&&+ \sum_{r=0}^K\mathbb{E}_\xi  A_{n,r}+\sum_{r=K+1}^{n-1}\mathbb{E}_\xi  A_{n,r},
\end{eqnarray}
where
$$ A_{n,r}:=\frac{1}{P_n^2}\sum_{w\in\T_r}\sum_{\substack{wi,wj\in\T_1(w)\\i\neq j}}\sum_{wi\tilde{u}\in\T_{n-r-1}(wi)}\sum_{wj\tilde{v}\in\T_{n-r-1}(wj)}f_{\xi,n}(X(wi\tilde{u}))f_{\xi,n}(X(wj\tilde{v})),$$
and $K=K(\xi)(\geq n_0)$ is  suitable large.
By Lemma \ref{LLL3.2} and condition (H2), for $n\geq n_0$,
\begin{eqnarray*}
\frac{1}{P_n^2}\mathbb{E}_\xi\left[\sum_{u\in\T_n}f_{\xi,n}^2(X(u))\right]&=&\frac{1}{P_n}\mathbb{E}_\xi f_{\xi,n}^2(Y_n)\\
&=&\frac{1}{P_n}\nu_\xi Q_0\cdots Q_{n-1}f_{\xi,n}^2\\
&\leq &\frac{\alpha_n}{P_n}\rightarrow0\;\; a.s.\quad\text{as $n\rightarrow\infty$.}
\end{eqnarray*}
Again by Lemma \ref{LLL3.2},  for every $w \in \T_r$, $i\ne j \in \mathbb N$, denoting $n_r:=n-r-1$ in the computation below,
\begin{eqnarray*}
&&B(wi,wj)\\ && \qquad :=\mathbb{E}_\xi\bigg[\sum_{wi\tilde{u}\in\T_{n_r}(wi)}f_{\xi,n}(X(wi\tilde{u}))\sum_{wj\tilde{v}\in\T_{n_r}(wj)}f_{\xi,n}(X(wj\tilde{v}))\bigg|\mathcal{F}_{r+1},X(wi),X(wj)\bigg] \\
&&\qquad =\mathbb{E}_{\xi,X(wi)}\bigg[\sum_{wi\tilde{u}\in\T_{n_r}(wi)}f_{\xi,n}(X(wi\tilde{u}))\bigg]
\mathbb{E}_{\xi,X(wj)}\bigg[\sum_{wj\tilde{v}\in\T_{n_r}(wj)}f_{\xi,n}(X(wj\tilde{v}))\bigg] \\
&&\qquad =\mathbb{E}_{\xi,X(wi)}\left[f_{\xi,n}(Y_{n_r}(wi))\right]\mathbb{E}_{\xi}\left[N_{n_r}(wi)\right]\mathbb{E}_{\xi,X(wj)}\left[f_{\xi,n}(Y_{n_r}(wj))\right]\mathbb{E}_{\xi}\left[N_{n_r}(wj)\right]\\
&&\qquad =m_{r+1}^2\cdots m_{n-1}^2Q_{r+1,n} f_{\xi,n}(X(wi))Q_{r+1,n} f_{\xi,n}(X(wj)),
\end{eqnarray*}
with the notation $Q_{r,n}=Q_{r}\cdots Q_{n-1}$.
Thus, 
\begin{eqnarray*}
&&\mathbb{E}_\xi  A_{n,r} \\
&& = \frac{1}{P_n^2}\mathbb{E}_\xi\left[\sum_{w\in\T_r}\sum_{\substack{wi,wj\in\T_1(w)\\i\neq j}} B(wi,wj)\right] \\
&& =\frac{1}{P_{r+1}^2}\mathbb{E}_\xi\left[\sum_{w\in\T_r}\sum_{\substack{wi,wj\in\T_1(w)\\i\neq j}} Q_{r+1,n} f_{\xi,n}(X(wi)) Q_{r+1,n} f_{\xi,n})(X(wj))\right]\\
&& = \frac{1}{P_{r}m_r^2}\nu_\xi Q_0\cdots Q_{r-1}J_{T^r\xi}(Q_{r+1,n} f_{\xi,n}\otimes Q_{r+1,n}f_{\xi,n}).
\end{eqnarray*}
For $r\geq K+1$, by conditions (H1) and (H3), 
\begin{eqnarray*}
\sum_{r=K+1}^{n-1}\mathbb{E}_\xi  A_{n,r}&\leq& \sum_{r=K+1}^\infty\frac{1}{P_{r}m_r^2}\nu_\xi Q_0\cdots Q_{r-1}J_{T^r\xi}(g\otimes g)\\&\leq&\sum_{r=K+1}^\infty\frac{\beta_r}{P_{r}m_r^2}\rightarrow0\;\; a.s. \quad\text{as $K\rightarrow\infty$. }
\end{eqnarray*}
It remains to consider $0\leq r\leq K$. For almost all $\xi$, for each $r$ fixed, by (H4), 
$$Q_{r+1,n}f_{\xi,n}\otimes Q_{r+1,n} f_{\xi,n}(y,z)\stackrel{n\rightarrow\infty}{\longrightarrow}0 \;\;\text{for each $(y,z)\in\mathcal X^2$.}$$
Notice that by (H1) and (H3), for $n\geq n_0(\xi)$,
\begin{eqnarray*}
&&\nu_\xi Q_0\cdots Q_{r-1}J_{T^r\xi}(Q_{r+1,n} f_{\xi,n}\otimes Q_{r+1,n} f_{\xi,n})\\
&& \qquad \qquad \qquad \qquad \qquad \qquad \leq\nu_\xi Q_0\cdots Q_{r-1}J_{T^r\xi}(g\otimes g)\leq \beta_r\qquad a.s.
\end{eqnarray*}
By the dominated convergence theorem, $\sum_{r=0}^{K}\mathbb{E}_\xi  A_{n,r}\rightarrow0$ a.s. as $n\rightarrow\infty$. The proof is completed.
\end{proof}

In particular, applying Proposition \ref{LLP3.3} with $f_{\xi,n}=f_\xi-\pi(f_\xi)$, we obtain the following result. We denote 
$$F_{\xi,n}=\nu_\xi Q_0\cdots Q_{n-1}.$$

\begin{pr}\label{LLP3.4}
%\textbf{Proposition 4.2.}
Let $\nu_\xi$ be the distribution of $X(\emptyset)$.  We assume that  for almost all $\xi$, there exist  a function $g$, an integer $n_0=n_0(\xi)$  and non negative numbers $(\alpha_n,\beta_n)$$=(\alpha_n(\xi),\beta_n(\xi))$ such that
\begin{itemize}
\item[(H1)] for all $n\geq n_0$ and   $x\in\mathcal{X}$,  
$$\sup_{0\leq r<n}Q_r\cdots Q_{n-1}|f_\xi|(x)\leq g(x);$$ 
\item[(H2)] for every $n\geq n_0$, $ F_{\xi,n}(f_\xi^2)\leq \alpha_n$, and
$ \alpha_n/P_n\rightarrow0 $ $(n\rightarrow\infty)$.
\item[(H3)] for every $n\in\mathbb N$, 
$$ \max\{F_{\xi,n}(J_{T^{n}\xi}(g\otimes g)),F_{\xi,n}(J_{T^{n}\xi}(g\otimes
\mathbf{1})),F_{\xi,n}(J_{T^{n}\xi}(\mathbf{1}\otimes \mathbf{1}))\}\leq
\beta_n,\quad\sum_n\frac{\beta_n}{{P_n}m_n^2}<\infty;$$
\item[(H4)]there exists $\pi (f_\xi)\in\R$ bounded by some constant $M$ such that for each $r$ fixed, 
$$\lim_{n\rightarrow\infty}Q_r\cdots Q_{n-1}f_\xi(x)=\pi (f_\xi)\quad\text{ for every $x\in\mathcal{X}$.}$$
\end{itemize}
Under (H1)-(H4), if additionally $W_n\rightarrow W$ in $\mathbb{P}_\xi$-$L^2$, then  we have for almost all $\xi$, 
$$\frac{\sum_{u\in \T_n} f_\xi(X(u))}{P_n}\rightarrow \pi (f_\xi)W\qquad \text{in $\mathbb{P}_\xi$-$L^2$.}$$
\end{pr}

In particular, if $f_\xi=f$, then conditions (H1) and (H4) can be simplified to
\begin{itemize}
\item[(H1')]$ Q_0\cdots Q_{n-1}|f|(x)\leq g(x)$ for all $n\in\mathbb{N}$ and  $x\in\mathcal{X}$;
\item[(H4')]there exists $\pi (f)\in\R$ such that
$Q_0\cdots Q_{n-1}f(x)\rightarrow \pi (f)$
 as $n\rightarrow\infty$ for every $x\in\mathcal{X}$.
 \end{itemize}

\medskip

\begin{proof}[Proof of Theorem \ref{LLNT3.1'}]We apply Proposition \ref{LLP3.3} with $f_{\xi,n}(x)=\mathbf{1}_{A}(x)-\mu_{\xi,n}(A)$. Indeed,  for $n$ large enough, $f_{\xi,n}$ is bounded by some constant $M$. We can let  $g=M$, $\alpha_n=M^2$ and $\beta_n=\mathbb{E}_{T^n\xi}N^2$. The assumptions
(\ref{assBPRE})  ensure that a.s.,
$$\frac{\alpha_n}{P_n}\rightarrow0,\quad
\sum_n \frac{\beta_n}{P_n m_n^2}<\infty\quad \text{and}\quad
W_n\rightarrow W \;\;\text{ in $\mathbb{P}_\xi$-$L^2$.}$$
Then Proposition \ref{LLP3.3} yields  (\ref{LLE3.2b}). Moreover  $\mathbb E_\xi W=1$ implies  that $q(\xi)=\p_\xi(W=0)<1$, so  $\{W>0\}=\{N_n\rightarrow\infty\}$  $\p_\xi$-a.s.. 
Finally, (\ref{LLE3.3b}) comes  from  (\ref{LLE3.2b}) and the fact that $W_n\rightarrow W>0$ a.s. on the non-extinction event $\{N_n\rightarrow\infty\}$.
\end{proof}

\begin{proof}[Proof of Corollary \ref{LLT3.1}]
We apply Proposition \ref{LLP3.4} with $f(x)=\mathbf{1}_{A}(x)$.   
Take $g=1$, $\alpha_n=1$ and $\beta_n=\mathbb{E}_{T^n\xi}N^2$.  
\end{proof}

\subsection{Adaptation of the proof for backward  law of large numbers}
%Q:this part is  not clear, the modification of Lemma 5.1 which I suggest would be a way to clarify it no ?\\\textbf{R:No, I didn't use Lemma 5.1 in this part except (5.5); here I just used the formula (5.4) with the tree notation changed. } VB : yes, but it seems that Lemma 5.1 could be rewritten so that in could be useful both for the proof of the forward and backward (and whole tree) LLN.\\
%\textbf{?????}
  
\begin{proof}[Proof of Corollary \ref{LLT3.2b} (\ref{LLNTBE1}) ]
Let $f_\xi(x)=\mathbf1_A(x)-\mu_\xi(A)$. Clearly, $f_\xi$ is bounded by $1$. Like the proof of Proposition \ref{LLP3.3}, we still have (\ref{LLNEsp}) with $f_{\xi}$ in place of $f_{\xi,n}$, but here and throughout this proof,  $\T_k(u)$ ($|u|+k\leq n$) denotes the set of individuals in $k$th generation of a tree rooted at $u$ in the environment $\xi^{(n)}$. Firstly, we have a.s.,
$$\frac{1}{P_n^2}\mathbb{E}_{\xi}\left[\sum_{u\in\T_n}f_{\xi}^2(X(u))\right]\leq \frac{1}{P_n}\rightarrow0\;\;\text{as $n\rightarrow\infty$.}$$ 
By Lemma \ref{LLL3.2}, for $|u|<n$ and $k\leq n-|u|$,
\begin{equation}
\frac{\mathbb{E}_{\xi,x}\sum_{uv\in\T_k(u)}f_\xi(X(uv))}{\mathbb{E}_{\xi^{(n)}}N_k(u)}=Q_{n-|u|-1}\cdots Q_{n-|u|-k}f_\xi(x).
\end{equation}
Therefore,
\begin{eqnarray*}
\mathbb{E}_{\xi}  A_{n,r}=(m_{n-1}\cdots m_{n-r})^{-2}   \mathbb{E}_{\xi}\left[\sum_{w\in\T_r}\sum_{\substack{wi,wj\in\T_1(w)\\i\neq j}}g_{\xi,n_r}(X(wi)) g_{\xi,n_r}(X(wj))\right],
\end{eqnarray*}
where $g_{\xi,n}(x)=Q_{n-1}\cdots Q_0f_\xi(x)$ and $n_r:=n-r-1$. 
For $r\geq K+1$, since $m_0>a$ and $\mathbb{E}_\xi N^2\leq(\mathbb{E}_\xi N^p)^{2/p}\leq b^{2/p}$, we get
\begin{eqnarray*}
\sum_{r=K+1}^{n-1}\mathbb{E}_{\xi}  A_{n,r}&\leq&\sum_{r=K+1}^{n-1}(m_{n-1}\cdots m_{n-r}m_{n_r}^2)^{-1}\mathbb{E}_{T^{n_r}\xi}N^2\\
&\leq&\sum_{r=K+1}^{\infty}a^{-(r+2)}b^{2/p}\rightarrow 0\;\;\text{as $K\rightarrow\infty$.}
\end{eqnarray*}
Now consider $0\leq r\leq K$. The fact that $m_0>a$ and $g_{\xi,n}$ is bounded by $1$ yields 
\begin{eqnarray*}
\mathbb{E}_{\xi}  A_{n,r}&\leq&a^{-2r}\mathbb{E}_{\xi}  \left[\sum_{w\in\T_r}N(w)\sum_{wi\in\T_1(w)}|g_{\xi,n_r}(X(wi))|\right]\\
&\leq&a^{-2r}\mathbb{E}_{\xi}  \left[N_{r+1}^{(n)}\sum_{w\in\T_{r+1}} |g_{\xi,n_r}(X(w))|\right]\\
&=&S_{n,r}+T_{n,r},
\end{eqnarray*}
where
$$S_{n,r}=a^{-2r}\mathbb{E}_{\xi}  \left[N_{r+1}^{(n)}\mathbf1_{\{N_{r+1}^{(n)}\leq K\}}\sum_{w\in\T_{r+1}} |g_{\xi,n_r}(X(w))|\right]$$
and
$$T_{n,r}=a^{-2r}\mathbb{E}_{\xi}  \left[N_{r+1}^{(n)}\mathbf1_{\{N_{r+1}^{(n)}>K\}}\sum_{w\in\T_{r+1}} |g_{\xi,n_r}(X(w))|\right].$$
We first deal with $S_{n,r}$ and fix $r \geq 0$. Let us label  the individuals in generation  $r+1$ as $1,2, \cdots$. For $j>N_{r+1}^{(n)}$,  the value of $g_{\xi,n_r}(X(j))$ will be $0$ by convention. So 
$$ S_{n,r}\leq Ka^{-2r}\sum_{j=1}^K\mathbb{E}_{\xi}   |g_{\xi,n_r}(X(j))|.$$
By (\ref{LLET3.2}), for almost all $\xi$,
$$g_{ \xi,n}(x)=\mathbb{P}_{\xi^{(n)},x}(Y_n\in A)-\mu_\xi(A)\stackrel{n\rightarrow\infty}{\longrightarrow}0\;\;\text{for every $x\in\mathcal{X}$.}$$
By the bounded convergence theorem, we have for each $r$ fixed, $\lim_n S_{n,r}=0$ a.s.
On the other hand, 
%\marginpar{\emph{There is no precise place. The  steps of proof are similar, but not the same. Moreover, this paper is awaiting large modification. }}
following arguments in Huang \& Liu \cite{huang1},  the fact that 
$$\mathbb{E}_\xi\left(\frac{N}{m_0}\right)^p<b/a^p$$ for some $p>2$ implies that for all $n$, $$\sup_{0\leq r\leq n}\mathbb{E}_{\xi^{(n)}}\left[W_{r+1}^{(n)}\right]^p\leq C_p$$ for some constant $C_p$ depending on $p$. Thus for each $r$ fixed and $\delta\in(0, p-2)$,
\begin{eqnarray*}
T_{n,r}&\leq&a^{-2r}\mathbb{E}_{\xi^{(n)}}  \left[\left(N_{r+1}^{(n)}\right)^2\mathbf1_{\{N_{r+1}^{(n)}>K\}}\right]\\
&=&K^{-\delta}a^{-2r}\mathbb{E}_{\xi^{(n)}}  \left[N_{r+1}^{(n)}\right]^{2+\delta}\\
&\leq&C_{r,\delta}K^{-\delta} \longrightarrow 0
\end{eqnarray*}
as $K\rightarrow\infty$, where $C_{r,\delta}$ is a constant depending on $r$ and $\delta$. Therefore, a.s., $\lim_{n}\mathbb{E}_{\xi}  A_{n,r}=0$ for each $r$ fixed, so that $ \sum_{0\leq r\leq K}\mathbb{E}_{\xi}  A_{n,r}\rightarrow 0$ as $n\rightarrow\infty$. So (\ref{LLNTBE1}) is proved.
\end{proof}

\begin{proof}[Proof of Corollary \ref{LLT3.2b}  (\ref{LLNeT3.2})]
 Since  $\sup_n\mathbb{E}_{\xi^{(n)}}\left[W_{n}^{(n)}\right]^2\leq C$ for some constant $C$, by the bounded convergence theorem, (\ref{LLNTBE1}) implies that 
\begin{equation}\label{LLNPBE2}
\frac{Z_n^{(n)}(A)-\mu_\xi(A)N_n^{(n)}}{P_n}\rightarrow 0  \;\; \text{in    $L^2$ and in $\mathbb{P}$-probability. }
\end{equation}
For any $\varepsilon>0$ and $\alpha>0$,
\begin{eqnarray}\label{LLNPBE3} &&\mathbb{P}\left(\bigg|\frac{Z_n^{(n)}(A)}{N_n^{(n)}}-\mu_\xi(A)\bigg|>\varepsilon, N_n^{(n)}>0\right)\nonumber\\
&\leq&\mathbb{P}\left(\bigg|\frac{Z_n^{(n)}(A)-\mu_\xi(A)N_n^{(n)}}{P_n}\bigg|>\alpha\varepsilon\right)+\mathbb{P}(0<W_n^{(n)}<\alpha).
\end{eqnarray}
By (\ref{LLNPBE2}), it is clear that the first term in the right side of the inequality above tends to 0  as $n\rightarrow\infty$.
When the environment is reversible in law, under $\mathbb{P}$, $W_n^{(n)}$ has the same distribution as $W_n$ of the forward case. Thus we have 
$$\mathbb{P}(0<W_n^{(n)}<\alpha)=\mathbb{P}(0<W_n<\alpha).$$
In the forward case, it is known that $W_n$ tends to a limit $W$ a.s., and
$$\lim_{n\rightarrow\infty}\mathbb{P}(N_n=0)=\mathbb{P}(N_n=0\; \text{for some $n$})=\mathbb{P}(W=0)\;\;a.s.$$
when $\mathbb{E}\log m_0>0$ and $\mathbb{E}\frac{N}{m_0}\log^+N<\infty$. Therefore we have 
$$\limsup_{n\rightarrow\infty}\mathbb{P}(0<W_n<\alpha)=\mathbb{P}(0<W\leq\alpha).$$
Since $\lim_{\alpha\downarrow0}\mathbb{P}(0<W\leq\alpha)=0$, for any $\eta>0$, there exists  $\alpha>0$ small enough such that $\mathbb{P}(0<W\leq\alpha)\leq \eta$. Taking the superior limit in (\ref{LLNPBE3}), we obtain
$$\limsup_{n\rightarrow\infty}\mathbb{P}\left(\bigg|\frac{Z_n^{(n)}(A)}{N_n^{(n)}}-\mu_\xi(A)\bigg|>\varepsilon, N_n^{(n)}>0\right)\leq \mathbb{P}(0<W\leq\alpha)\leq \eta.$$
Letting $\eta\rightarrow0$ completes the proof. 
\end{proof}

\subsection{Proof of the law of large numbers on the whole tree}
Proposition \ref{LLNPW} below is in the same vein as Proposition \ref{LLP3.4}. We recall the notation $F_{\xi,n}=\nu_\xi Q_0\cdots Q_{n-1}$.

\begin{pr}\label{LLNPW}
Let $\nu_\xi$ be the distribution of $X(\emptyset)$. We assume that  for almost all $\xi$, there exist non negative numbers $(\alpha_n,\beta_n)$$=(\alpha_n(\xi),\beta_n(\xi))$ such that
\begin{itemize}
\item[(H1)] for all   $n\in\mathbb{N}$ and $x\in\mathcal{X}$,  
$$\sup_{0\leq r<n}Q_r\cdots Q_{n-1}|f_\xi|(x)\leq g(x);$$ 
%\item[(H1)] there exists a function $g$ such that $\sup_{0\leq r<n}Q_r\cdots Q_{n-1}|f_\xi|(x)\leq g(x)$ for all $n\in\mathbb{N}$ and $x\in\mathbb{R}$;
\item[(H2)] for every $n\in\mathbb N$,
$$ \max\{F_{\xi,n}(f_\xi^2),\;F_{\xi,n}(|f_\xi| g),F_{\xi,n}(g)\}\leq \alpha_n, \quad \sum_n\frac{\alpha_n}{P_n}<\infty ;$$
\item[(H3)] for every $n\in\mathbb N$, 
$$ \max\{F_{\xi,n}(J_{T^{n}\xi}(g\otimes g)),F_{\xi,n}(J_{T^{n}\xi}(g\otimes
\mathbf{1})),F_{\xi,n}(J_{T^{n}\xi}(\mathbf{1}\otimes \mathbf{1}))\}\leq
\beta_n,\quad\sum_n\frac{\beta_n}{P_nm_n^2}<\infty;$$
\item[(H4)]there exists $\pi (f_\xi)\in\R$ bounded by some constant $M$ such that for each $r$ fixed, 
$$\lim_{n\rightarrow\infty}\frac{1}{n-r}\sum_{k=r+1}^nQ_r\cdots
Q_{k-1}f_\xi(x)=\pi (f_\xi)\quad \text{for every $x\in\mathcal{X}$.}$$
\end{itemize}
%\marginpar{This notation $\mathbb{G}_k$ is confusing with rest of the literature, in particular \cite{guy,delmar}, since it is the generation $n$ for them, can you change it  ?\textbf{OK}}
%Under (H1)-(H4), if additionally $\frac{1}{n}\sum_{k=1}^nW_k\rightarrow W$ in $\mathbb{P}_\xi$-$L^2$, then  we have for almost all $\xi$, 
%$$\frac{1}{n} \sum_{u\in \mathbf\Gamma_k} \frac{f_\xi(X(u))}{P_{|u|}}\rightarrow \pi (f_\xi)W\qquad \text{in $\mathbb{P}_\xi$-$L^2$,}$$
%where $\mathbf\Gamma_k=\bigcup_{k=1}^n\T_k$.
\end{pr}

In particular, if $f_\xi=f$, then conditions (H1) and (H4) can be simplified to
\begin{itemize}
\item[(H1')] $ Q_0\cdots Q_{n-1}|f|(x)\leq g(x)$ for  all $n\in\mathbb{N}$ and $x\in\mathcal{X}$;
\item[(H4')]there exists $\pi (f)\in\R$ such that
$\frac{1}{n}\sum_{k=1}^nQ_0\cdots Q_{k-1}f(x)\rightarrow \pi (f)$
 as $n\rightarrow\infty$ for every $x\in\mathcal{X}$.
\end{itemize}

\medskip
 \begin{proof}[Proof of Proposition \ref{LLNPW}]We only prove the case where $\pi(f_\xi)=0$. For general case, it suffices to consider $f_\xi-\pi(f_\xi)$ in place of $f_\xi$. We shall prove that under the hypothesis (H1)-(H4), a.s.,
 \begin{equation} \mathbb{E}_\xi\left(\frac{1}{n}\sum_{u\in\mathbf\Gamma_n}\frac{f_\xi(X(u))}{P_{|u|}}\right)^2\rightarrow0\;\;\text{as $n\rightarrow\infty$.}
 \end{equation}
 Notice that $\mathbf\Gamma_n=\bigcup_{k=1}^n\T_k$ is the set of all individuals in the first $n$ generation. For $u,v\in\mathbf\Gamma_n$, we discuss for two cases: (i) $u$ and $v$ in the same life line, which means that one is an ancestor of the other, i.e. $u\wedge v=u\;\text{or}\; v$; (ii) the contrary case, i.e. $u\wedge v\neq u,v$. So we can write 
 \begin{eqnarray*} \mathbb{E}_\xi\left(\frac{1}{n}\sum_{u\in\mathbf\Gamma_n}\frac{f_\xi(X(u))}{P_{|u|}}\right)^2&=&\frac{1}{n^2}\mathbb{E}_\xi\left[\sum_{u,v\in\mathbf\Gamma_n}\frac{f_\xi(X(u))f_\xi(X(v))}{P_{|u|}P_{|v|}}\right]\\
 &=& S_{n,\xi}+T_{n,\xi},
 \end{eqnarray*}
 where $$S_{n,\xi}=\frac{1}{n^2}\mathbb{E}_\xi\left[\sum_{\substack{u,v\in\mathbf\Gamma_n\\u\wedge v=u \;or\; v}}\frac{f_\xi(X(u))f_\xi(X(v))}{P_{|u|}P_{|v|}}\right],$$
 and
 $$ T_{n,\xi}=\frac{1}{n^2}\mathbb{E}_\xi\left[\sum_{\substack{u,v\in\mathbf\Gamma_n\\u\wedge v\neq u, v}}\frac{f_\xi(X(u))f_\xi(X(v))}{P_{|u|}P_{|v|}}\right].$$
 We need to prove that $S_{n,\xi}$, $T_{n,\xi}$ tend to 0 a.s. as $n\rightarrow\infty$.
 
At first,   $S_{n,\xi}$ can be decomposed as 
\begin{equation}\label{LLNTWP1} S_{n,\xi}=\frac{2}{n^2}\mathbb{E}_\xi\left[\sum_{\substack{u,v\in\mathbf\Gamma_n\\u<v}}\frac{f_\xi(X(u))f_\xi(X(v))}{P_{|u|}P_{|v|}}\right]+\frac{1}{n^2}\mathbb{E}_\xi\left[\sum_{u\in\mathbf\Gamma_n}\frac{f^2_\xi(X(u))}{P_{|u|}^2}\right].
\end{equation}
For the second term in the right side of the equality above, we use  (H2) and get 
\begin{eqnarray*}
\frac{1}{n^2}\mathbb{E}_\xi\left[\sum_{u\in\mathbf\Gamma_n}\frac{f^2_\xi(X(u))}{P_{|u|}^2}\right]&=&
\frac{1}{n^2}\sum_{r=1}^n\mathbb{E}_\xi\left[\sum_{u\in\mathbb{T}_r}\frac{f^2_\xi(X(u))}{P_{|r|}^2}\right]\\
&=&\frac{1}{n^2}\sum_{r=1}^n P_r^{-1}\nu_\xi Q_0\cdots Q_{r-1} f_\xi^2\\
&\leq&\frac{1}{n^2}\sum_{r=1}^\infty P_r^{-1}\alpha_r\rightarrow0 \;\;\text{as $n\rightarrow\infty$,}
\end{eqnarray*}
since $\sum_n P_n^{-1}\alpha_n<\infty$ a.s.. For $u,v\in\mathbf \Gamma_n$ with $u<v$, we write $v=u\tilde v$ with $|\tilde v|=k-r$ if $u\in \T_r$ and $v\in \T_k$ ($k>r$). Then, using  
(H2) and the a.s. convergence of $\sum_n P_n^{-1}\alpha_n$ again, we have a.s.,
\begin{eqnarray*}
&&\frac{2}{n^2}\mathbb{E}_\xi\left[\sum_{\substack {u,v\in\mathbf\Gamma_n\\u<v}}\frac{f_\xi(X(u))f_\xi(X(v))}{P_{|u|}P_{|v|}}\right]\\
&& \quad = \frac{2}{n^2}\sum_{r=1}^n\sum_{k=r+1}^n\mathbb{E}_\xi\left[\sum_{u\in\mathbb{T}_r}\sum_{u\tilde v\in\mathbb{T}_{k-r}(u)}\frac{f_\xi(X(u))f_\xi(X(u\tilde v))}{P_{|r|} P_{|k|}}\right]\\
&& \quad = \frac{2}{n^2} \sum_{r=1}^n\sum_{k=r+1}^n P_r^{-1}P_k^{-1}\mathbb{E}_\xi\left[\sum_{u\in\mathbb{T}_r}f_\xi(X(u))
\mathbb{E}_{\xi,X(u)}\left[\sum_{u\tilde v\in\mathbb{T}_{k-r}(u)}f_\xi(X(u\tilde v))\right]\right]\\
&& \quad = \frac{2}{n^2} \sum_{r=1}^n\sum_{k=r+1}^n P_r^{-2}\mathbb{E}_\xi\left[\sum_{u\in\mathbb{T}_r}f_\xi(X(u) )Q_r\cdots Q_{k-1} f_\xi(X(u))\right]\\
&& \quad \leq \frac{2(n-r)}{n^2} \sum_{r=1}^nP_r^{-1}\nu_\xi Q_0\cdots Q_{r-1}|f_\xi|g\\
&& \quad \leq \frac{2}{n}\sum_{r=1}^\infty P_r^{-1}\alpha_r\rightarrow0 \;\;\text{as $n\rightarrow\infty$.}
\end{eqnarray*}
Hence we have $S_{n,\xi} \rightarrow 0$ a.s. as $n\rightarrow\infty$.

Now we consider $T_{n,\xi}$. For $K=K(\xi)$ fixed suitable large, 
$$T_{n,\xi}=\sum_{r=1}^K\mathbb{E}_\xi A_{n,r}+\sum_{r=K+1}^{n-1}\mathbb{E}_\xi A_{n,r},$$
where 
$$A_{n,r}=\frac{1}{n^2}\sum_{k,l=r+1}^n\sum_{w\in\T_r}\sum_{\substack{wi,wj\in\T_1(w)\\i\neq j}}\sum_{\substack{wi\tilde u\in\T_{k-r-1}(wi)\\wj\tilde v\in\T_{k-r-1}(wj)}}\frac{f_\xi(X(wi\tilde u))f_\xi(X(wi\tilde v))}{P_k P_l}.$$
With the notation $Q_{r,n}=Q_{r}\cdots Q_{n-1}$, we compute
\begin{eqnarray*}
&&\mathbb{E}_\xi A_{n,r}\\&& \quad =\frac{1}{n^2}\sum_{k,l=r+1}^n(P_k P_l)^{-1}\mathbb{E}_\xi\left[\sum_{w\in\T_r}\sum_{\substack{wi,wj\in\T_1(w)\\i\neq j}}\mathbb{E}_{\xi,X(wi)}\left[\sum_{wi\tilde u\in\T_{k-r-1}(wi)}f_\xi(X(wi\tilde u))\right]\right. \\
&&\qquad \qquad\qquad\qquad\qquad\qquad\qquad\times\quad\mathbb{E}_{\xi,X(wi)}\left.\left[\sum_{wj\tilde v\in\T_{l-r-1}(wj)}f_\xi(X(wj\tilde v))\right]\right]\\
&& \quad = \frac{1}{n^2}\sum_{k,l=r+1}^n P_{r+1}^{-2}\mathbb{E}_\xi\left[\sum_{w\in\T_r}\sum_{\substack{wi,wj\in\T_1(w)\\i\neq j}} Q_{r+1,k}f_\xi(X(wi)) Q_{r+1,l}f_\xi(X(wj))\right]\\
&& \quad = \frac{(n-r)^2}{n^2}P_{r+1}^{-2}\mathbb{E}_\xi\left[\sum_{w\in\T_r}\sum_{\substack{wi,wj\in\T_1(w)\\i\neq j}}R_{n,r}(X(wi))R_{n,r}(X(wj)) \right]\\
&&\quad = \frac{(n-r)^2}{n^2}\frac{1}{P_r m_r^2}\nu_\xi Q_0\cdots Q_{r-1}J_{T^{r+1}\xi}(R_{r,n}\otimes R_{r,n}),
\end{eqnarray*}
where 
$$R_{n,r}(x)=\frac{1}{n-r}\sum_{k=r+1}^n Q_{r+1,k}f_\xi(x).$$
we oberve that thanks to  (H1), $\sup_{0\leq r<n}|R_{r,n}|\leq g$ for every $n$.
And by (H4), for almost all $\xi$, for each $r$ fixed,  
$$R_{n,r}\otimes R_{n,r}(y,z)\stackrel{n\rightarrow\infty}{\longrightarrow}0 \;\;\text{for each $(y,z)\in\mathcal X^2$.}$$
Following similar arguments in the proof of Proposition \ref{LLP3.3}, we see $T_{n,\xi}\rightarrow0$ a.s. as $n\rightarrow\infty$. The proof is completed.
 \end{proof}
 
 \begin{lem}\label{LLNLL}
 Let $p>1$. If $\mathbb{E}(\log m_0)>0$ and $\mathbb{E}\left(\log\mathbb{E}_\xi\left(\frac{N}{m_0}\right)^p\right)<\infty$, then 
 $$\lim_{n\rightarrow\infty}\frac{1}{n}\sum_{k=1}^{n}W_k=W\qquad\text{in $\mathbb{P}_\xi$-$L^p$.}$$
 \end{lem}
 
 \begin{proof}
 By Theorems 2.1 and 2.2 in Huang \& Liu \cite{huang1}, the integrability assumptions 
%conditions $\mathbb{E}(\log m_0)>0$ and $\mathbb{E}\left(\log\mathbb{E}_\xi\left(\frac{N}{m_0}\right)^p\right)<\infty$ 
ensure that a.s., 
 $$0<\mathbb{E}_\xi W^p<\infty \qquad\text{ and}\qquad \lim_{k\rightarrow\infty}\rho^k\mathbb{E}_\xi|W_k-W|^p=0,$$
for some $\rho>1$. It  implies that 
\begin{equation}\label{LLNEL4.2}
\mathbb{E}_\xi|W_k-W|^p\leq C_\xi\rho^k\;\;a.s.
\end{equation}
for some $C_\xi<\infty$. Therefore, a.s.,
\begin{eqnarray*}
\mathbb{E}_\xi\left|\frac{1}{n}\sum_{k=1}^nW_k-W\right|^p&=&\mathbb{E}_\xi\left|\frac{1}{n}\sum_{k=1}^n(W_k-W)\right|^p\\
&\leq&\frac{1}{n}\sum_{k=1}^n\mathbb{E}_\xi|W_k-W|^p\\
&\leq& C_\xi\frac{1}{n}\sum_{k=1}^\infty\rho^{-k}\rightarrow0 \;\;\text{as $n\rightarrow\infty$,}
\end{eqnarray*}
since $\sum_{k=1}^\infty\rho^{-k}<\infty$.
 \end{proof}
 
\begin{proof}[Proof of Theorem \ref{LLT3.3}]
By Lemma \ref{LLNLL}, $\frac{1}{n}\sum_{k=1}^{n}W_k\rightarrow W$ in $\mathbb{P}_\xi$-$L^2$. Applying Proposition \ref{LLNPW} with $f(x)=\mathbf1_A(x)$, $g=1$, $\alpha_n=1$ and $\beta_n=\mathbb{E}_{T^n\xi}N^2$, we obtain (\ref{LLNEW1}).

%\marginpar{\textbf{Proof added}}
Now we prove (\ref{LLNEW2}). Since $W_n\rightarrow W$ 	a.s.,   for any $\delta>0$, there exists $n_0$ such that $\forall n\geq n_0$,
\begin{equation}\label{LLNPW1}
|W_n-W|<\delta\qquad\text{and}\qquad \left|\frac{1}{W_n}-\frac{1}{W}\right|<\delta
\end{equation}
a.s. on the non-extinction event. We write 
\begin{eqnarray}\label{LLNPW2}
&&\frac{1}{n}\sum_{k=1}^{n}\frac{Z_k(A)}{N_k}-\mu(A)\nonumber\\
&& \quad =\frac{1}{n}\sum_{k=1}^{n_0}\frac{Z_k(A)}{P_k}\left(\frac{1}{W_k}-\frac{1}{W}\right)+\frac{1}{n}\sum_{k=n_0+1}^{n}\frac{Z_k(A)}{P_k}\left(\frac{1}{W_k}-\frac{1}{W}\right)\nonumber\\&& 
\qquad +\frac{1}{W}\left(\frac{1}{n}\sum_{k=1}^{n}\frac{Z_k(A)}{P_k}-\mu(A)W\right).
\end{eqnarray}
Obviously,  the first term in the right hand side of (\ref{LLNPW2}) tends to $0$ a.s. on the non-extinction event as $n$ goes to infinity. And the convergence of the third term is from (\ref{LLNEW1}). Therefore, to prove (\ref{LLNEW2}), we only need to show that the second term 
\begin{equation}\label{LLNPW3}
\frac{1}{n}\sum_{k=n_0+1}^{n}\frac{Z_k(A)}{P_k}\left(\frac{1}{W_k}-\frac{1}{W}\right)\stackrel{n\rightarrow\infty}{\longrightarrow}0\qquad\text{in $\mathbb P_\xi$-probability}
\end{equation}
on the non-extinction event for almost all $\xi$. In fact, by (\ref{LLNPW1}), we have
$$\left|\frac{1}{n}\sum_{k=n_0+1}^{n}\frac{Z_k(A)}{P_k}\left(\frac{1}{W_k}-\frac{1}{W}\right)\right|<\delta(W+\delta)$$
a.s. on the non-extinction event. The arbitrariness of $\delta$ yields (\ref{LLNPW3}).
 \end{proof}

\subsection{Proof of the central limit theorem}
\begin{proof}[Proof of Theorem \ref{LLT3.5}]
We shall apply Proposition \ref{LLP3.3} with $f_{\xi,n}(x)=\mathbf{1}_{A_n}(x)-\Phi(y)$, where $A_n=(-\infty,b_n(\xi)y+a_n(\xi) ]$.
 By (\ref{LLET3.5}) and Dini's Theorem, we have a.s.,
\begin{equation}\label{LLE3.10}
\lim_{n\rightarrow\infty}\sup_{y\in\mathbb{R}}\left|\mathbb{P}_{\xi,x}\left(\frac{Y_n-a_n(\xi)}{b_n(\xi)}\leq y\right)-\Phi(y)\right|=0  \qquad\text{ for every $x\in\mathcal{X}$.}
\end{equation}
Notice that $|f_{\xi,n}|\leq1$. Take  $g=1$, $\alpha_n=1$ and $\beta_n=\mathbb{E}_{T^n\xi}N^2$. It is easy to verify that (H1)-(H3) are satisfied. For (H4), by (\ref{LLE3.10}) and the continuity of $\Phi$, for each $r$ fixed,
\begin{eqnarray*}
&&Q_r\cdots Q_{n-1}f_{\xi,n}(x)\\
& =&\mathbb{P}_{T^r\xi,x}(Y_{n-r}\leq b_n(\xi)y+a_n(\xi))-\Phi(y)\\
& \leq& \sup_{y\in\mathbb R}\left|\mathbb{P}_{T^r\xi,x}\left(\frac{Y_{n-r}-a_{n-r}(T^r\xi)}{b_{n-r}(T^r\xi)}\leq y\right)
-\Phi\left(y\right)\right|\\
&& +\left|\Phi\left( \frac{b_n(\xi) y+a_n(\xi)-a_{n-r}(T^r\xi)}{b_{n-r}(T^r\xi)}\right)-\Phi(y)\right|,
\end{eqnarray*}
which goes to $0$ as $n\rightarrow\infty$ for every $x\in\mathcal{X}$.
By Proposition \ref{LLP3.3}, a.s.,
$$\frac{Z_n( A_n)}{P_n}-\Phi(y)\frac{N_n}{P_n}\rightarrow 0\qquad \text{in $\mathbb{P}_\xi$-$L^2$.}$$
Since $W_n\rightarrow W$ a.s. and in $\mathbb{P}_\xi$-$L^2$, we get (\ref{LLE3.8}) and (\ref{LLE3.9}).
\end{proof}

\textbf{Acknowledgement.} 
This work  was partially  funded by  Chair Modelisation Mathematique et Biodiversite VEOLIA-Ecole Polytechnique-MNHN-F.X., the professorial chair Jean Marjoulet, the project MANEGE `Mod\`eles
Al\'eatoires en \'Ecologie, G\'en\'etique et \'Evolution'
09-BLAN-0215 of ANR (French national research agency) and the National Natural Science Foundation of China, Grant No. 11101039.

%\marginpar{to check and complete}

\end{document}